\documentclass[final]{siamltex}
\pdfoutput=1

\usepackage{graphicx}
\usepackage{amsmath}
\usepackage{amssymb}
\usepackage{amsfonts}
\usepackage{epstopdf}
\usepackage{setspace} 
\usepackage{xspace} 
\usepackage{bm}         
\usepackage{braket}
\usepackage{booktabs}
\usepackage[numbers]{natbib}
\usepackage[hypertexnames=false,colorlinks=true,linkcolor=blue,citecolor=blue,
            bookmarksopen=false,bookmarks=false,
            pdfstartview=XYZ,pdffitwindow=true,pdfcenterwindow=true]{hyperref}
\usepackage{multirow}

\newtheorem{dfn}{Definition}[section]
\newtheorem{lem}{Lemma}[section]
\newtheorem{rem}{Remark}[section]
\newtheorem{hyp}{Hypothesis}
\newtheorem{res}{Result}

\newcommand{\PIN}{\textnormal{{\tiny PIN}}}

\newcommand{\IAA}{\textnormal{{\tiny IAA}}}

\newcommand{\neigh}{\mathcal{N}}
\newcommand{\vy}{\bm{y}}
\newcommand{\vys}{\bm{y}^*}
\newcommand{\va}{\bm{a}}
\newcommand{\vas}{\bm{a}^*}
\newcommand{\vp}{\bm{p}}
\newcommand{\vps}{\bm{p}^*}
\newcommand{\mD}{\bm{D}}
\newcommand{\vnu}{\boldsymbol{\nu}}
\newcommand{\vpi}{\boldsymbol{\pi}}
\newcommand{\vdelta}{\boldsymbol{\delta}}
\newcommand{\vpsi}{\boldsymbol{\psi}}
\newcommand{\vphi}{\boldsymbol{\varphi}}
\newcommand{\veta}{\boldsymbol{\eta}}
\newcommand{\vxi}{\boldsymbol{\xi}}
\newcommand{\R}{\mathbb{R}}
\newcommand*{\etal}{et al.\ }
\DeclareMathOperator{\spec}{Spec}
\begin{document}

\author{
  Delphine Draelants\footnotemark[3]\ \footnotemark[4]\ \footnotemark[1]
  \and
  Daniele Avitabile\footnotemark[3]\ \footnotemark[4]\ \footnotemark[2]
  \and
  Wim Vanroose\footnotemark[1]
}

\title{Localised auxin peaks in concentration-based transport models of
the shoot apical meristem}

\renewcommand{\thefootnote}{\fnsymbol{footnote}}

\footnotetext[3]{These authors contributed equally and should be considered as
joint first authors.}

\footnotetext[4]{Corresponding authors: 
                 \texttt{delphine.draelants@uantwerpen.be},
                 \texttt{Daniele.Avitabile@nottingham.ac.uk} }

\footnotetext[1]{Department of Mathematics and Computer Science, Universiteit
Antwerpen, Middelheimlaan 1, B 2020, Antwerpen}

\footnotetext[2]{Centre for Mathematical Medicine and Biology, School of
Mathematical Sciences, University of Nottingham, University Park, Nottingham,
NG7 2RD, UK}

\renewcommand{\thefootnote}{\arabic{footnote}}

\maketitle

\begin{abstract}
  We study the formation of auxin peaks in a generic class of
  concentration-based auxin transport models, posed on static plant tissues.
  Using standard asymptotic analysis we prove that, on bounded domains, auxin
  peaks are not formed via a Turing instability in the active transport
  parameter, but via simple corrections to the homogeneous steady state. When the
  active transport is small, the geometry of the tissue encodes the
  peaks' amplitude and location: peaks arise where cells have fewer neighbours,
  that is, at the boundary of the domain. We test our theory and perform
  numerical bifurcation analysis on two models which are known to generate auxin
  patterns for biologically plausible parameter values. In the same parameter
  regimes, we find that realistic tissues are capable of generating a
  multitude of stationary patterns, with a variable number of auxin peaks, that
  can be selected by different initial conditions or by quasi-static changes
  in the active transport parameter. The competition between active transport
  and production rate determines whether peaks remain localised or cover the
  entire domain. We relate the occurrence of localised patterns to a
  \textit{snaking} bifurcation structure, which is known to arise in a wide
  variety of nonlinear media but has not yet been reported in plant models.
\end{abstract}


\section{Introduction}
The hormone auxin plays a crucial role in plant development~\cite{Scarpella2006,
Band2012a, Lavenus2013, deWit2014}, yet the mechanisms through which it
accumulates in certain cells and interacts with cell growth mechanisms
remain largely unclear. The patterns formed during the growth of a plant are
controlled by the local auxin concentration. For example, it
is known that the distribution of auxin maxima in the shoot apex gives
rise to the formation of primordia~\cite{Smith2006,
deReuille2006,Jonsson2006,Kuhlemeier2007, Benkova2003, Heisler2005}.  
Similarly, the distribution of auxin in the root tip
coordinates cell division and cell
expansion~\cite{Grieneisen2007,Jones2008}. In models of root hair
initiation, intra-cellular levels and gradients of auxin concentration
influence the localisation of G-proteins, which in turn promote hair
formation~\cite{Brena2013}. In addition, it is known that the
distribution of auxin in the leaf primordia mediates vascular
patterning~\cite{Scarpella2006}.
In recent years, many aspects of the molecular basis of these mechanisms have been
unraveled and mathematical models of auxin transport have been
proposed to explain growth and development~\cite{Berleth2007, Stoma2008,
Kramer2006, Wabnik2011}.

Computer simulations are often used to compare the model output with
observed data such as auxin distribution, venation patterns, growth or
development. At cellular level, carriers such as PIN-FORMED
(PIN) proteins, which are localised in the cell membrane, determine the
rate and direction of auxin transport. The coordinated activity of
many cells can create peaks of auxin that drive differentiation and
growth.  Various models that implement and refine these ideas have
been proposed~\cite{Smith2006, deReuille2006, Jonsson2006, Grieneisen2007,
Swarup2005, Wabnik2010, Merks2007}. Such models differ primarily in the specifics of
the transport and the coupling to the cell growth and division, but a common feature
is that they generate spatially-extended patterns of auxin concentration,
which have also been observed experimentally.

Existing transport models can be classified into two main categories,
\textit{flux-based} and \textit{concentration-based}, depending on how
auxin influences the localization of transport mediators (PINs) to
form patterns. In flux-based models, first proposed
in~\cite{Mitchison1981}, the polarization depends on the net auxin
flux between neighbouring cells: the higher the net flux towards the
neighbours, the more PIN will accumulate at the membrane, and changes
in the PIN distribution determine changes in auxin fluxes. By
contrast, in concentration-based models it is assumed that the PIN
accumulation on the membrane is caused by differences in auxin
concentration between neighbouring cells. This type of models was
introduced in~\cite{Smith2006} and~\cite{Jonsson2006}. For other reviews on
flux- and concentration-based models we refer the reader to~\cite{Krupinski2010,
Kramer2008, Rolland-Lagan2005,Band2012b}.

In the models mentioned above, patterns are found by direct numerical
simulation, upon choosing control parameters within a plausible
biological range. However, there is still a large uncertainty on many
of the parameter values which are often approximated~\cite{Feugier2005,
Kramer2011}, adopted from different systems~\cite{Steinacher2012} or estimated
with large error margins~\cite{Heisler2006}. Furthermore, it is unclear what is
the effect of systematic parameter variations on the generated patterns and how
this relates to the behaviour of the biological system: understanding the
formation of auxin peaks from a dynamical system standpoint is still an open
problem, therefore a mathematical exploration of the parameter variations may
generate new, experimentally testable hypotheses, thereby gaining insight into
pattern formation mechanisms~\cite{Scarpella2006,
Benkova2003, Bilsborough2011}.

In spite of the uncertainty on experimental parameters, it is believed that
active transport is a key player in auxin
patterning~\cite{Kramer2011}. Transport models posses an
inherent time-scale separation: the growth hormone dynamics involve short time
scales (of the order of seconds) \cite{Brunoud2012}, while changes in cellular
shapes and proliferation of new cells occur on much slower time scales (hours or
days) \cite{Beemster1998}. In order to determine the distribution of auxin in
the plant, it is then possible to concentrate on the fast time scale of the
hormone transport, assume a static cell structure and study the plant tissue as
a dynamical system, subject to variations in the active transport parameter.

In this paper we perform such exploration on concentration-based auxin models,
which are studied using standard bifurcation analysis
techniques~\cite{Krauskopf}. In particular, we find steady states of the 
system and explore their dependence upon control parameters,
investigating how patterns lose or gain stability in response to
parameter changes. The aim is to predict qualitatively the distribution patterns
that can occur for a certain parameter range and to understand transitions
between different types of patterns.

At present, only a few papers regard transport models as dynamical
systems: among them, Reference~\cite{vanberkel2013} stands out for being a
\textit{systematic} analysis of flux- and concentration-based models, whose
auxin patterns are studied by considering local interactions between cells. In
the present paper, we take the analysis one step further, by finding steady
states simultaneously in the whole tissue, and by studying the important effects
of its boundedness on the auxin patterns.

The main result of our analysis is that, in a large class of concentration-based
models posed on finite tissues, peaks do not arise from an instability of the
homogeneous steady state, as it was previously reported for unbounded
tissues~\cite{Smith2006, Jonsson2006, Sahlin2009, Newell2008}: on regular
bounded domains, the geometry of the tissue drives the formation of small auxin
peaks, which nucleate without instabilities
near the boundary.

Further, we investigate the effects of changes in the active transport
coefficient, in the diffusivity coefficient and in the auxin production coefficient for two
specific examples: the concentration-based model proposed by
Smith \etal \cite{Smith2006} and a more recent modification studied by Chitwood
\etal \cite{Chitwood2012}. In these systems the localised peaks, determined
analytically for small values of the active transport coefficient, persist also
for moderate and large values of this control parameter. We found that, owing to
their boundedness, realistic tissues can select from a multitude of patterned
configurations, characterised by a variable number of localised peaks and
organised in a characteristic snaking bifurcation diagram.

Snaking bifurcation diagrams are commonly found for localised states~ 
arising in (systems of) nonlinear partial differential equations posed in
one~\cite{knobloch:08,woods-champneys:99,burke-knobloch:07,
  burke-knobloch:07c,beck-knobloch-etal:09,chapman-kozyreff:09},
two~\cite{lloyd-sandstede-etal:08,avitabile-lloyd-etal:10} and
three~\cite{beaume:13,jacono2013} spatial dimensions, as well as in
discrete~\cite{Taylor2010a} and nonlocal systems~\cite{coombes-lord-etal:03,
laing-troy:03,rankin2014,Avitabile2014}. However, this mechanism has never
been reported for auxin models: solutions with localised peaks undergo a
series of saddle-node bifurcations, giving rise to a hierarchy of steady
states with an increasing number of bumps.
A direct consequence of this mechanism is that the resulting patterns are robust
against changes in the transport parameter and other control parameters as
found, for instance, by Sahlin \etal~\cite{Sahlin2009}. We argue that this
mechanism could be a robust feature in several other types of
concentration-based auxin models.

The paper is structured as follows: in Section~\ref{sec:resultsSummary} we
summarise our working hypotheses and describe the main results of the paper; in
Section~\ref{subsec:TheExistenceOf} we present our calculations for a simple 1D
tissue, giving a primer on bifurcation analysis for auxin models; in
Section~\ref{subsec:2DDom} we generalise our results to 2D tissues, which are
further discussed in Section~\ref{sec:discussion}; finally, we provide a more
formal presentation of our general asymptotic results in
Section~\ref{sec:MaterialsAndMethods}.

\section{Mathematical formulation and summary of the main results}
\label{sec:resultsSummary}
\begin{figure}
  \centering
  \includegraphics[width=\textwidth]{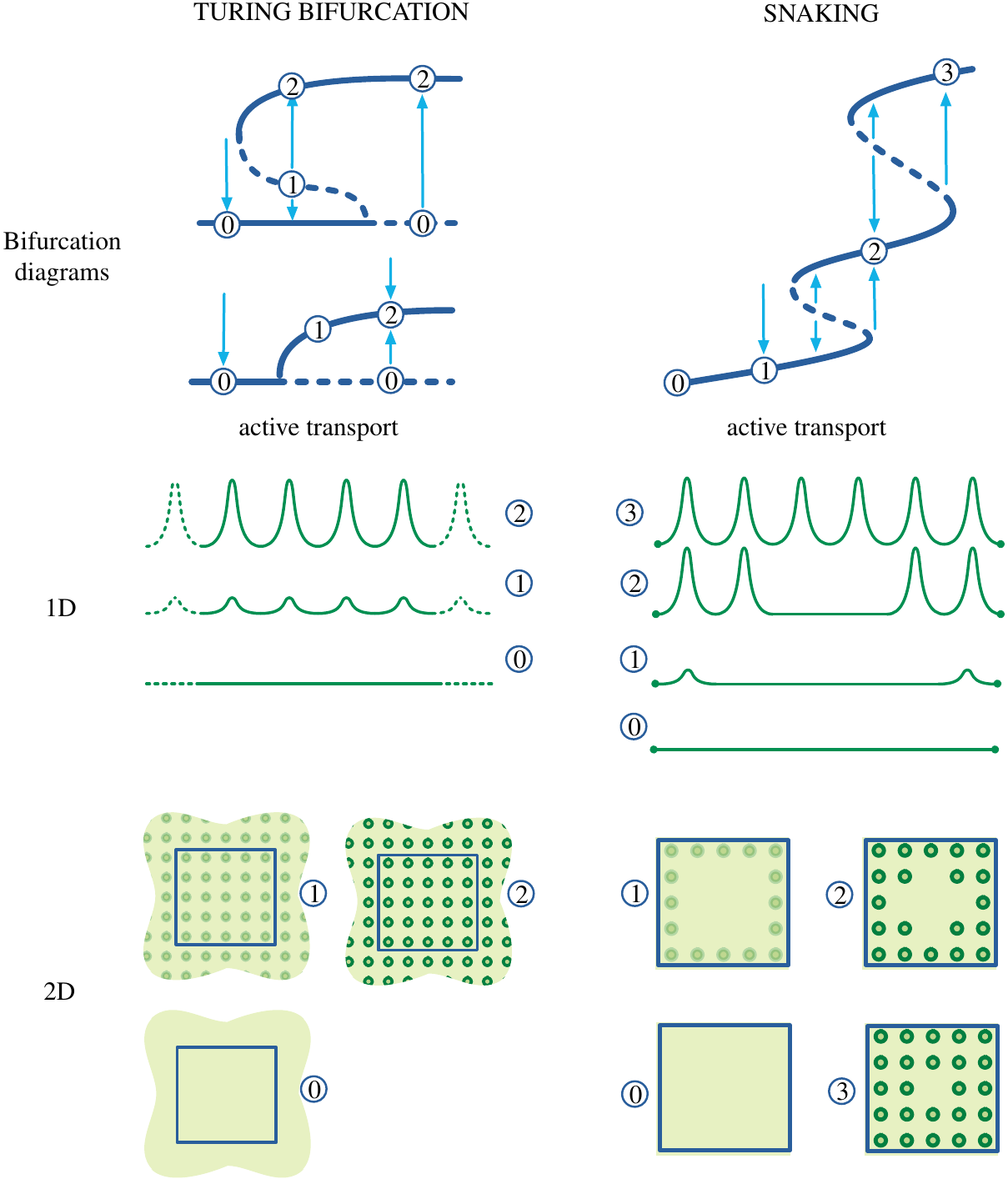}
  \caption{Two mechanisms for the formation of spots, explained with bifurcation
  diagrams (top) and representative patterns in 1D (centre) and 2D (bottom).
  Turing bifurcation (left): on unbounded domains, or bounded domains with
  periodic boundary conditions, the homogeneous state (0) is attracting for low
  values of the active transport parameter $T$ (a blue arrow indicates time
  evolution); when $T$ is increased above a critical value, the flat state
  becomes unstable and the system evolves towards a fully patterned state (2);
  states with low auxin peaks (1) may be attracting or repelling, depending on
  the bifurcation type. Snaking mechanism (right): on bounded domains with free
  boundary conditions, the flat state (0) exists only for $T=0$; upon increasing
  $T$, we find an attracting state with low auxin peaks at the boundary (1),
  followed by several other attracting states with variable number of peaks (2,
  3); the tissue is able to select from different patterns, depending on the
  initial condition and on the value of $T$.}
  \label{fig:sketch1D}
\end{figure}
We begin by giving a generic definition of concentration-based models, and a
summary of the main results of the paper. In concentration-based models, cells
are identified with an index $i \in \{1,\ldots,n\}$ and to each cell it is
associated a set of neighbours $\mathcal{N}_i$, containing $\vert \mathcal{N}_i
\vert$ neighbours, and a vector of $m$ time-dependent state variables
$\vy_i(t)$. For instance, $\vy_i$ may contain the auxin concentration ($m=1$) or
both auxin and PIN-FORMED1 concentrations ($m=2$). Generically, the rate of
change of the concentrations is expressed as a balance between production and
decay within the cell, diffusion towards neighbouring cells and active
transport, hence we study generic concentration-based auxin models of the form
\begin{equation}
\dot{\vy}_i = \vpi(\vy_i) - \vdelta(\vy_i) + \mD \sum_{j \in \neigh_i} (\vy_j - \vy_i)
  + T \sum_{j \in \neigh_i} \vnu_{ji}(\vy_1,\ldots,\vy_n) - \vnu_{ij}(\vy_1,\ldots,\vy_n),
 \label{eq:model}
\end{equation}
where $\vpi, \vdelta$, are the production and decay functions, respectively,
$\mD$ is a diffusion matrix, $T$ is the active transport parameter and
$\vnu_{ij}$ are the active transport functions. In this paper we concentrate on
the fast time scale of hormone transport and hence consider plant organs as
static cell structures, so the number of cells $n$ remains constant in time. We
make two key assumptions:
\begin{hyp}[Regular domains]\label{hyp:domain}
  Cells are arranged in a regular domain, that is, they
  have all the same shape and size and they tessellate the tissue. We do not make any
  assumption on the dimensionality of the domain.
\end{hyp}
\begin{hyp}[Active transport functions]\label{hyp:nu}
The active transport functions can be expressed as
\begin{equation}
 (\vnu_{ij} )_l  = 
 \psi_l(\vy_i,\vy_j)
 \frac{
  \varphi_l(\vy_j) }
 {
   \sum_{k \in \neigh_i} \varphi_l (\vy_k)
 }, \quad \textrm{for $l=1,\ldots,m$},
 \label{eq:hyp}
\end{equation}
where the functions $\psi_l$, $\varphi_l$ depend on the model choices.
\end{hyp}

Hypothesis~\ref{hyp:nu} is a factorisation of the active transport functions that is
met by several models in literature~\cite{Smith2006, Jonsson2006, Bayer2009,
Heisler2006, Chitwood2012}: active transport between
cell $i$ and $j$ is influenced by the respective concentrations $\vy_i$ and
$\vy_j$, but also by concentrations in the neighbours of cell $i$.
In Section~\ref{subsec:twomodels} we introduce examples of concrete models
satisfying Hypothesis~\ref{hyp:nu} and in Section~\ref{sec:MaterialsAndMethods}
we derive explicit expressions for the
corresponding functions $\psi_l$, $\varphi_l$.

It is established in literature that concentration-based models are capable of
reproducing auxin patterns that are found in SAM experiments~\cite{Smith2006,
Jonsson2006, deReuille2006, Jonsson2010, Chitwood2012, Sahlin2009}, for
biologically realistic values of the transport parameters. In this paper we will
address the following
questions: What type of patterns are generated by the class of models described
above? Do they all predict the occurrence of auxin peaks? What are the instabilities
that lead to the formation of auxin peaks? Are auxin patterns robust to changes in
the control parameters and initial conditions?

These question have been partially addressed in previous papers~\cite{Smith2006,
Jonsson2006, Chitwood2012, Sahlin2009, deReuille2006, Farcot2013, vanberkel2013},
where analytical results have been obtained only for particular models and only for
regular domains without boundaries, where all cells have the same number of
neighbours. In such domains a homogeneous steady state $\vy_i = \vy^*$, satisfying
the balancing condition $\vpi(\vy_i) = \vdelta(\vy_i)$, is known to exist for all
values of $T$~\cite{Smith2006, deReuille2006, Jonsson2006, Chitwood2012, Sahlin2009,
Farcot2013, Newell2008}. In domains without boundaries, the formation of peaks
has been explained in terms of a Turing bifurcation in the active transport parameter
(see Figure~\ref{fig:sketch1D}): peaks are formed all at once as $T$ is varied, and
they derive from an instability of the homogeneous steady state. From a dynamical
system viewpoint, however, we expect that boundary conditions and finite domain sizes
affect the formation and
selection of auxin patterns. The main result of our investigation is that the
mechanism for the formation of peaks is radically different in tissues of finite
size, and in particular:
\begin{res}[Homogeneous steady states]\label{res:HomState}
  In finite domains, concentration-based models~\eqref{eq:model} satisfying
  Hypotheses~\ref{hyp:domain}--\ref{hyp:nu} support a
  homogeneous equilibrium $\vy_i = \vy^*$ for $T=0$ but, generically, this
  homogeneous state is not present when $T\neq 0$. This result is a direct
  consequence of the geometry of the domain: an inspection of
  Equation~\eqref{eq:hyp} shows that the active transport terms $\vnu_{ij}$
  contain a nested sum over the neighbours of the cell $i$ and, in a finite
  domain, the number of neighbours varies from cell to cell, namely cells at the
  boundary of the regular tissue have fewer neighbours than cells in the
  interior; consequently, 
  $\vnu_{ij} - \vnu_{ji}$ is generally different from $\mathbf{0}$ at the
  boundary, even when $\vy_i = \vy^*$. The conclusion is
  that, on finite domains, peaks can not form with a Turing bifurcation, since
  the homogeneous steady state exists for $T=0$, but not for $T\neq 0$ (however
  small). For regular domains where the number of neighbours is the same for all
  cells, the Turing mechanism is still possible. In addition, a Turing
  bifurcation is also possible in any domain, provided that $T=0$ and diffusion
  is used as a bifurcation parameter. In this regime, however, the active
  transport is absent and the tissue is a standard medium with
  reaction-diffusion mechanisms, which is not a biologically valid hypothesis
  for auxin models.
\end{res}
\begin{res}[Origin of auxin peaks]\label{res:smallT}
  To understand how peaks are formed from the homogeneous state
  we study the case of small active transport coefficients. For $0 < T \ll 1$ and in
  the absence of passive transport, $\mD=\mathbf{0}$, the models above
  generate steady states with
  small peaks. Such states take the form
  \begin{equation}
    \vy_i = \vy^* 
          +  T \xi_i \big[\vpi'(\vy^*) - \vdelta'(\vy^*)\big]^{-1}
          \vpsi(\vy^*,\vy^*) + \mathcal{O}(T^2)
  \label{eq:pert}
  \end{equation}
  where the coefficients $\xi_i$ depend on the number of neighbours at
  distance 2 from the $i$th cell\footnote{Neighbours at distance 2 from a cell
  $i$ are neighbours of the neighbours of cell $i$.}, namely
  \[
     \xi_i = 1-\sum_{j \in \neigh_i} \frac{1}{\vert \neigh_j \vert}.
  \]
  Equation~\eqref{eq:pert} predicts that peaks are formed as small perturbations
  of the homogeneous steady state (see Figure~\ref{fig:sketch1D}).
  The amplitude of small peaks is proportional to $T$ and $\xi_i$. Importantly,
  $\xi_i=0$ in the interior of
  regular domains, therefore peaks localise at the boundary and without
  bifurcations, as opposed to the Turing scenario where they form everywhere
  owing to an instability of $\vy^*$: the mechanism on finite domains is purely
  geometric, as the location of the peaks is determined by the factors $\xi_i$.
\end{res}
\begin{res}[Effect of passive transport]\label{res:diffusion}
  For small active transport and at the presence of passive transport,
  $\mD \neq \mathbf{0}$,
  we still obtain solutions with localised peaks, similar to the case discussed
  above. The location of the peaks depends again on $\xi_1, \ldots, \xi_n$.
\end{res}

\begin{rem}[Applicability of analytical results]\label{rem:generality}
  Results~\ref{res:HomState}--\ref{res:diffusion} are valid for generic models of
  the form~\eqref{eq:model}, provided they satisfy
  Hypotheses~\ref{hyp:domain}--\ref{hyp:nu}. In particular, these results are
  valid for regular cellular arrays in any spatial dimension. The main implication
  of this result is that a wide class of concentration-based models are able to
  generate spontaneously auxin peaks in various geometries, irrespective of the
  model specifics. A similar derivation can be done for irregular domains,
  albeit the coefficients $\xi_i$ depend in this case on the local cellular
  volumes as well as the number of neighbours. 
\end{rem}

The analytical theory described above, which is presented in more detail in
Section~\ref{sec:MaterialsAndMethods}, is valid only for small values of the active
transport coefficient and does not explain the formation of peaks in the interior of
the domain~\cite{deReuille2006, Sahlin2009, Jonsson2012}. While it is difficult to
make general analytical predictions for larger values of $T$, it is possible to
explore the solution landscape of specific models via numerical methods. We
investigated regular domains in two concrete models by Smith
\etal~\cite{Smith2006} and Chitwood \etal~\cite{Chitwood2012} 
(henceforth called the Smith model and the Chitwood model, respectively)
which satisfy Hypotheses~\ref{hyp:domain}--\ref{hyp:nu} and will be described in
detail in Section~\ref{subsec:twomodels}. Our computations confirm the
theoretical Results~\ref{res:HomState}--\ref{res:diffusion} and provide
numerical evidence for the following conclusions:

\begin{res}[Formation of stable auxin spots in the interior]\label{res:snaking}
   As the active transport rises, the two models by Smith and Chitwood predict
   the formation of peaks in the interior. Peaks are formed
   progressively, from the boundary towards the interior via saddle-node
   bifurcations, with a characteristic snaking bifurcation diagram. From a
   biological perspective, this means that the tissue can form peaks that are
   robust with respect to changes in the control parameters. In addition, the
   tissue is capable of selecting from a variety of auxin patterns, with a variable
   number of spots, depending on the initial conditions and on the value of the auxin
   transport parameter $T$ (see Figure~\ref{fig:sketch1D}).
\end{res}
\begin{res}[Robustness of the snaking mechanism]\label{res:robustness}
   The scenario above is robust to perturbations to secondary parameters, that
   is, solutions with localised peaks at the boundary should be observable in
   experiments for which the two models are applicable, if the active transport
   is inhibited with respect to passive transport. 
\end{res}

Furthermore, in numerical calculations we are able to violate
Hypotheses~\ref{hyp:domain}--\ref{hyp:nu} and see how this affects our results.
An important conclusion is the following:

\begin{res}[Irregular domains]\label{res:irregular}
  When the Smith and the Chitwood models are posed on irregular
  domains, the bifurcation structure presented above persists. Auxin peaks are
  formed progressively via saddle-node bifurcations, albeit they may in principle
  nucleate in the interior of the domain before the boundary is filled.
  Furthermore, this scenario is robust to changes in secondary parameters, so
  Result~\ref{res:robustness} is still valid on irregular domains for both
  models.
\end{res}

\subsection{Two models of auxin transport}
\label{subsec:twomodels}

We now present two models that will be analysed in detail using the asymptotic
theory of Section~\ref{sec:MaterialsAndMethods} and numerical simulations.
As a first example we consider the Smith model~\cite{Smith2006}, which features
2 state variables per cell, namely the indole-3-acetic acid (IAA)
concentration, $a_i(t)$, and the PIN-FORMED1 (PIN1) amount, $p_i(t)$. The model
features IAA production, decay, active and passive transport terms, whereas for
PIN1 only production and decay are included.
This results in the
following set of coupled nonlinear ODEs
\begin{align}
  \begin{split}
  \frac{da_i}{dt} =& 
		    \frac{\rho_\IAA}{1+\kappa_\IAA a_i} - \mu_\IAA a_i + \frac{D}{V_i}
		    \sum_{j \in \mathcal{N}_i} l_{ij} \big( a_j-a_i \big)
                   \\
                 & + \frac{T}{V_i}
		 \sum_{j \in \mathcal{N}_i} 
		   \Bigg[ 
		      P_{ji}(\va,\vp) \frac{a_j^2}{1+\kappa_T a_i^2} - 
		      P_{ij}(\va,\vp) \frac{a_i^2}{1+\kappa_T a_j^2}
		   \Bigg],
  \end{split}
  \label{eq:SmithModel1}
  \\
  \frac{dp_i}{dt} =& \frac{\rho_{\PIN_0} + \rho_\PIN a_i}{1+\kappa_\PIN p_i} - \mu_\PIN p_i,
  \label{eq:SmithModel2}
\end{align}
for $i=1,\ldots,n$. In this model $D$ is a diffusion coefficient, $V_i$ is the
cellular volume, $l_{ij}=S_{ij}/(W_i + W_j)$ is the ratio between
the contact area $S_{ij}$ of the adjacent cells $i$ and $j$, and the sum of the
corresponding cellular wall thicknesses $W_i$ and $W_j$.
In addition, $T$ is the active transport
coefficient and $P_{ij}$ is the number of PIN1 proteins on the
cellular membrane of cell $i$ facing cell $j$,
\begin{equation}
  P_{ij}(\va,\vp) = p_i 
	            \frac{ l_{ij} \exp{(c_1 a_j)} }{
		    \sum_{k \in \mathcal{N}_i}l_{ik} \exp{(c_1 a_k)}}.
  \label{eq:SmithModel3}
\end{equation}
The Smith model posed on a regular domain satisfies 
Hypotheses~\ref{hyp:domain}--\ref{hyp:nu} in
Section~\ref{sec:resultsSummary} and the reader can find explicit expressions
for the functions $\varphi_l$, $\psi_l$ in Section~\ref{sec:one-dimensional-2-comp}.
More details on the model and simulations of realistic phyllotactic patterns can
be found in~\cite{Smith2006}.

The second concentration-based transport model that will be studied
below is the more recent Chitwood model~\cite{Chitwood2012}.
This modification of the Smith model is able to produce stable spiral
phyllotactic patterns once cell division is included. The system also features 2
variables per cell, the IAA concentration and the PIN1
amount, and it is given by the following set of coupled nonlinear ODEs
\begin{align}
  \begin{split}
  \frac{da_i}{dt} =& 
		    \frac{\rho_\IAA}{1+\kappa_\IAA a_i} - \mu_\IAA a_i + \frac{D}{V_i}
		    \sum_{j \in \mathcal{N}_i} l_{ij} \big( a_j-a_i \big)
                   \\
                 & + \frac{T}{V_i}
		 \sum_{j \in \mathcal{N}_i} 
		   \Bigg[ 
		       P_{ji}(\va,\vp)
		       \frac{\exp{(c_2 a_j)}-1}{\exp{(c_2 a_i)}}
		       - P_{ij}(\va,\vp)
		       \frac{\exp{(c_2 a_i)}-1}{\exp{(c_2 a_j)}}
		   \Bigg],
  \end{split}
  \label{eq:ChitwoodModel1}
  \\
  \frac{dp_i}{dt} =& \frac{\rho_{\PIN_0} + \rho_\PIN a_i}{1+\kappa_\PIN p_i} -
  \mu_\PIN p_i, \label{eq:ChitwoodModel2}
\end{align}
for $i=1,\ldots,n$, where $P_{ij}$ are given by~\eqref{eq:SmithModel3} and the
only new parameter, $c_2$, controls the exponential transport. 
The Chitwood  model posed on a regular domain also satisfies
Hypotheses~\ref{hyp:domain}--\ref{hyp:nu}, as shown in
Section~\ref{sec:one-dimensional-2-comp}.
 
In the reminder of the paper we shall fix most parameters in the Smith 
and the Chitwood models, and we will examine variations in
the active transport parameter
$T$, auxin diffusion coefficient $D$ and auxin production coefficient
$\rho_\IAA$. In Table \ref{table:parameters} we report a brief description of
parameters for both models, together with characteristic values and units, which
are taken from~\cite{Smith2006,Chitwood2012}. 

\begin{table}
  \centering
  \caption{Control parameters for the Smith and Chitwood
  models (parameter values are taken from~\cite{Smith2006} and
  \cite{Chitwood2012}). We examine variations in $T$, $D$ and $\rho_{_{\IAA}}$, for
  which we report a range of values in the second part of the table.}
\begin{tabular}{cllccc}
  \toprule
\multirow{2}{*}{Symbol} & \multirow{2}{*}{Description}   & \multirow{2}{*}{Domain} & \multicolumn{2}{c}{Value} & \multirow{2}{*}{Unit} \\
       & 						 &                         & \multicolumn{1}{r}{Smith \etal} & Chitwood \etal & \\
  \midrule
$c_1$              & PIN distribution               &           & $1.099$ & $1.099$ & 1/$\mu$M     \\
$\kappa_{_{\PIN}}$ & PIN saturation                 &           & $1$ & $1$ & 1/$\mu$M     \\
$\kappa_T$         & Transport saturation           &           & $1$ &         &              \\
$c_2$              & Exponential transport          & 2D reg.   &         & $0.588$ & 1/$\mu$M     \\
	       	   &                                & 2D irreg. &         & $0.405$ & 1/$\mu$M     \\
$\kappa_{_{\IAA}}$ & IAA saturation                 &           & $1$ & $1$ & 1/$\mu$M     \\
$\rho_{_{\PIN_0}}$ & PIN base production            &           & $0$ & $0$ & $\mu$M/h     \\
$\rho_{_{\PIN}}$   & PIN production                 &           & $1$ & $1$ & 1/h          \\
$\mu_{_{\PIN}}$    & PIN decay                      &           & $0.1$ & $0.1$ & 1/h          \\
$\mu_{_{\IAA}}$    & IAA decay                      &           & $0.1$ & $0.1$ & 1/h          \\
  \midrule
$\rho_{_{\IAA}}$   & IAA production                 & 			& $\left[0.3 , 1.5\right]$ & $\left[0.7 , 2.0\right]$  & $\mu$M/h     \\
$D$                & IAA diffusion                  &           & $\left[0 , 1\right]$ & $\left[0 , 1\right]$ & $\mu$m$^2$/h \\
$T$                & IAA transport coefficient      & 1D reg.   & $\left[0 , 6\right]$ & & $\mu$m$^3$/h \\
                   &                                & 2D reg.   & $\left[0 , 2\right]$ & $\left[0 , 2.5\right]$ & $\mu$m$^3$/h\\ 
                   &                                     & 2D irreg. & $\left[0, 120\right]$ & $\left[0, 95\right]$ & $\mu$m$^3$/h\\
\bottomrule
\end{tabular}
\label{table:parameters}
\end{table}

\begin{rem}[Comparison with experimental parameters]\label{rem:pars}
  The model equations~\ref{eq:SmithModel1}--\ref{eq:ChitwoodModel2} show that
  the transport parameters $T$ and $D$ are scaled by the cellular
  volumes $V_i$: comparisons to experimental parameters and other computer
  simulations available in literature should be based on the ratios $T/V_i$ and
  $D/V_i$. Through these ratios we implicitly specify $T/D$, so as to account
  for competition between active and passive transports.
\end{rem}

\begin{rem}[Tissue types]\label{rem:tissueType}
  We model 3 plant tissues: 
  identical cubic cells arranged on a line of finite length (here and
  henceforth, 1D regular), identical hexagonal prismic cells tessellating a
  finite square (2D regular)  and irregular prismic cells tessellating an
  almost-circular domain (2D irregular, taken from Merks \etal~\cite{Merks2011}). We
  stress that the nomenclature 1D and 2D refers to the domain, not the cells,
  which are assumed to have consistently assigned volumes and contact areas. We
  note that cellular volumes and contact areas may change between different domains
  (see also Remark~\ref{rem:pars}).
\end{rem}

\section{Results}
\label{sec:Results}

\subsection{A primer on the formation of auxin peaks - a 1D regular tissue}
\label{subsec:TheExistenceOf}
\begin{figure}
  \centering
  \includegraphics{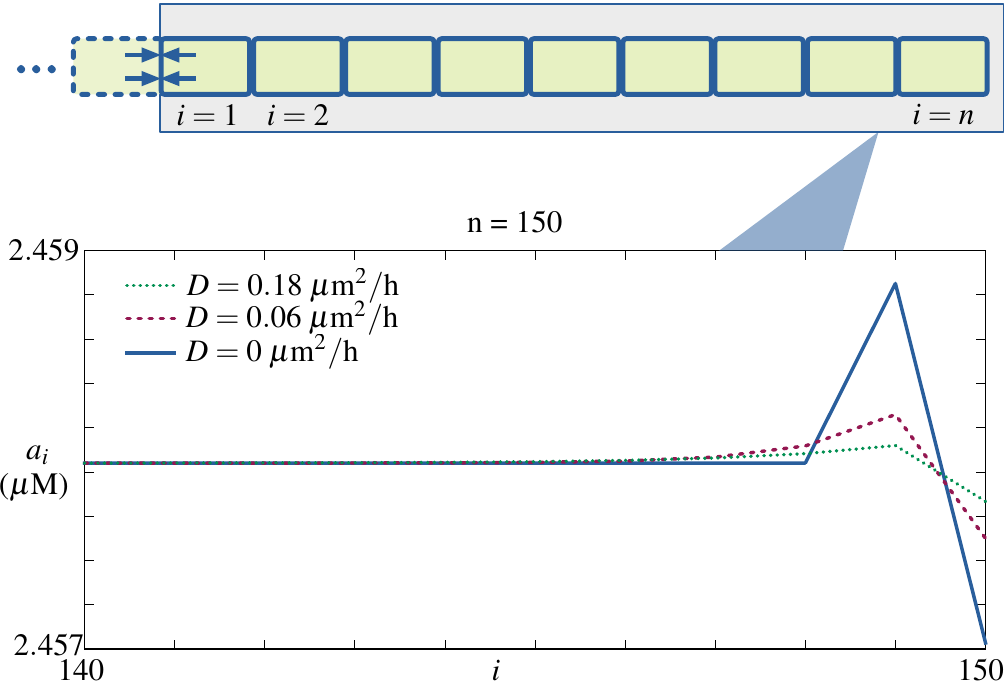}
  \caption{Approximate solution pattern in the 1D regular tissue. Top: geometry
  of a one-dimensional cellular array of identical cells; there exist cells to
  the left of $i=1$, but the next flux with cell $1$ is 0; at $i=n$ we have a
  physical boundary, therefore no cells exist to the right of cell $n$ and
  $\mathcal{N}_n = n-1$. Bottom: predicted approximate solution pattern $a^* + T
  \alpha_i$ in the proximity of the boundary $i=n$ for the Smith model
  with $n=150$ and various values of the diffusion coefficient. Parameters: $T =3
  \cdot 10^{-5}\mu\textrm{m}^3/\textrm{h}$, $D= 1\,\mu \textrm{m}^2/\textrm{h}$,
  $\rho_\IAA = 0.85 \, \mu\textrm{M}/\textrm{h}$
  other parameters as in Table \ref{table:parameters}.
  }
  \label{fig:FiniteDomainWithDiffusion}
\end{figure}

In this section, we present in detail
Results~\ref{res:HomState}--\ref{res:robustness} for the Smith model
posed on a 1D regular tissue, which is represented in
Figure~\ref{fig:FiniteDomainWithDiffusion}.
The tissue consists of a file of $n$ identical cubic cells with volume $V=1
\mu\textrm{m}^3$ and  $l_{ij}=1 \, \mu\textrm{m}$. We assume that there exist cells
to the left of cell $1$ and that the net proximal flux is zero, hence we prescribe
Neumann boundary conditions at $i=1$. We impose that there are no cells to the
right of cell $n$, so we use a free boundary condition\footnote{The free
boundary condition at $i=n$ should not be confused
with a zero Dirichlet boundary condition, for which we would prescribe
$\mathcal{N}_n=\{n-1,n+1\}$ and $\vy_{n+1}(t) \equiv 0$.} 
at $i=n$ by setting $\mathcal{N}_n=n-1$. In this geometry, the tissue has a physical
boundary only at $i=n$, as illustrated in Figure~\ref{fig:FiniteDomainWithDiffusion}, so each cell has
$2$ neighbours, except cell $n$, which has only $1$ neighbour. 
This geometry (with periodic boundary conditions) has
been studied in other auxin-patterning simulations: it was used
in reference~\cite{Farcot2013} as an approximation of a root tissue and
in reference~\cite{Draelants2013} to model part of a leaf, between the midvein
and the margin.
In this paper, we use this geometry only as a primer to illustrate 
how the asymptotic and numerical calculations are used to predict the
formation of auxin spots in concentration-based models. 
While the focus is on the Smith  model, several of
the results presented in this section, are valid in more general models and
spatial configurations (we refer the reader to Remark~\ref{rem:generality} and the
whole Section~\ref{sec:resultsSummary} for further comments on the generality of
these results).

\subsubsection{From homogeneous to patterned solutions}
\label{subsubsec:AFiniteOneDimensionalDomainWithoutDiffusion}
The Smith model posed on a regular domain satisfies 
Hypotheses~\ref{hyp:domain}--\ref{hyp:nu} (as shown in
Section~\ref{sec:one-dimensional-2-comp}), so we can apply the asymptotic theory in
Section~\ref{sec:MaterialsAndMethods} (in particular, Lemma~\ref{lemma1}).
On an unbounded array (or on a bounded array with periodic boundary conditions) every
cell has the same number of neighbours, therefore the model admits the following
homogeneous steady state
\begin{equation}
  \begin{aligned}
    & a_i = a^* = 
    \frac{-1 + \sqrt{1+4\kappa_\IAA \rho_\IAA /\mu_\IAA}}{2\kappa_\IAA}, \\
    & p_i = p^* = \frac{-1 + \sqrt{1+4\kappa_\PIN (\rho_{\PIN_0} + \rho_\PIN
    a^*)/\mu_\PIN}}{2\kappa_\PIN},
  \end{aligned}
  \label{eq:trivialSolution}
\end{equation}
for $i=1,\ldots n$. 
However the 1D regular domain considered here is finite, so the homogeneous
solution exists only for $T=0\; \mu \textrm{m}^3/\textrm{h}$, as for positive
$T$ the sums in the transport term in~\eqref{eq:SmithModel1} do not vanish in
general (Result \ref{res:HomState}). 
For $D=0\, \mu \textrm{m}^2/\textrm{h}$ and $ T \ll 1\, \mu \textrm{m}^3/\textrm{h}$,
Lemma~\ref{lemma1} shows the existence of patterns in the form of small deviations from
the homogeneous steady state (see Section~\ref{sec:one-dimensional-2-comp} for a
full derivation)
\begin{equation}
\begin{aligned}
a_i &= 
\begin{cases}
  a^*  & \text{for $i = 1 \ldots n-2$,} \\[1em]
  a^* + 
  T 
  \dfrac{p^*}{2V}
  \bigg[
  \dfrac{\rho_\IAA \kappa_\IAA}{1 + \kappa_\IAA (a^*)^2}
  + \mu_\IAA
  \bigg]^{-1}
  \dfrac{(a^*)^2}{1 + \kappa_T (a^*)^2}
  & \text{for $i = n-1$,} \\[1em]
  a^* -
  T
  \dfrac{p^*}{2V}
  \bigg[
  \dfrac{\rho_\IAA \kappa_\IAA}{1 + \kappa_\IAA (a^*)^2}
  + \mu_\IAA
  \bigg]^{-1}
  \dfrac{(a^*)^2}{1 + \kappa_T (a^*)^2}
  & \text{for $i = n$.}
\end{cases}\\
p_i &= 
  \begin{cases}
    p^*  & \text{for $i=1,\ldots,n-2$},\\
    p^* + \left[\dfrac{(\rho_{\PIN_0} + \rho_\PIN a^*)\kappa_\PIN}
	  {(1 + \kappa_\PIN p^*)^2}+ \mu_\PIN\right]^{-1}\!\!\!\dfrac{\rho_\PIN}{1+\kappa_\PIN p^*}(a_i-a^*) & \text{for $i=n-1$},\\
        p^* - \left[\dfrac{(\rho_{\PIN_0} + \rho_\PIN a^*)\kappa_\PIN}
	  {(1 + \kappa_\PIN p^*)^2}+ \mu_\PIN\right]^{-1}\!\!\!\dfrac{\rho_\PIN}{1+\kappa_\PIN p^*}(a_i-a^*) &\text{for $i=n$}.
  \end{cases}
\end{aligned}
\label{eq:asymptoticSolutions}
\end{equation}

The equations above predict that, for infinitesimal values of the transport
coefficient $T$, the perturbed solution coincides with the homogeneous solution,
except for a small peak at cell $n-1$ and a small dip at cell $n$ (see also bold
curve in Figure~\ref{fig:FiniteDomainWithDiffusion}). In other words,
peaks are present where the number of neighbours differs from the number of
neighbours in the unbounded domain, that is, where the sum in the active
transport in Equation~\eqref{eq:SmithModel1} is nonzero (Result \ref{res:smallT}). 
The analysis above applies also in the limit of slow dynamics of the PIN1
proteins: upon assuming $p$ constant and homogeneous in the tissue, we find that
$a_i$ is as in Equation~\ref{eq:asymptoticSolutions}, with $p^*$ replaced by $p$
(see Section~\ref{sec:one-dimensional} for details).

We have thus established that, in regular domains, a small auxin transport
coefficient $T$ elicit low auxin peaks. Such correlation was previously reported
in numerical experiments on various models~\cite{Jonsson2006,Sahlin2009} and we
now provide a mathematical explanation of this phenomenon.

Asymptotic calculations can also be carried out in the presence of
diffusion ($D \neq 0$), leading to a linear system for the perturbations. 
In Section~\ref{subsec:modelsWithDiffusion} we present a derivation for
generic models in generic tissues, which is then specialised for the Smith model
as an example. Quantitative results of this calculation are shown in
Figure~\ref{fig:FiniteDomainWithDiffusion}, where we plot approximate steady
states for the Smith model towards the boundary $i=n$ for $T=3\cdot
10^{-5} \mu \textrm{m}^3/\textrm{h}$ and various values of the diffusion
coefficient. We notice that, in the regime of small active transport and
comparatively
much bigger diffusion coefficient, a peak is still present at the
boundary. Inspecting the solid line ($D=0 \; \mu
\textrm{m}^2/\textrm{h}$) and the dashed lines ($D=0.06 \; \mu
\textrm{m}^2/\textrm{h}$ and $D=0.18 \; \mu \textrm{m}^2/\textrm{h}$)
we see that the peaks decrease in amplitude and are more spread out,
as expected (Result \ref{res:diffusion}).  
As a concluding remark, we point out that these results are not influenced
by the no-flux boundary conditions specified at $i=1$: the only physical boundary is
at $i=n$, where the number of neighbours differs from the interior.

\subsubsection{Forming high peaks in the interior via snaking}

\label{subsec:SnakingContinuationCurve}
We now turn to the more interesting question of how the tissue develops high
auxin peaks (Result \ref{res:snaking}) which are observed in experiments. Once
again, we illustrate our findings in the 1D regular case and generalise in the
following sections.

In realistic simulations the transport coefficient $T$ is not
necessarily small \cite{Smith2006,Jonsson2006,Kramer2011}, therefore it is
interesting to explore the solution landscape when $T$ is increased at the
presence of diffusion. This is done using numerical bifurcation analysis, that
is, equilibria of systems~\eqref{eq:model}
are followed in parameter space using Newton--Raphson method and pseudo-arclength
continuation~\cite{Krauskopf}.  Linear stability is then inferred computing the
spectrum of the Jacobian at the steady state.

\begin{figure}
  \centering
  \includegraphics{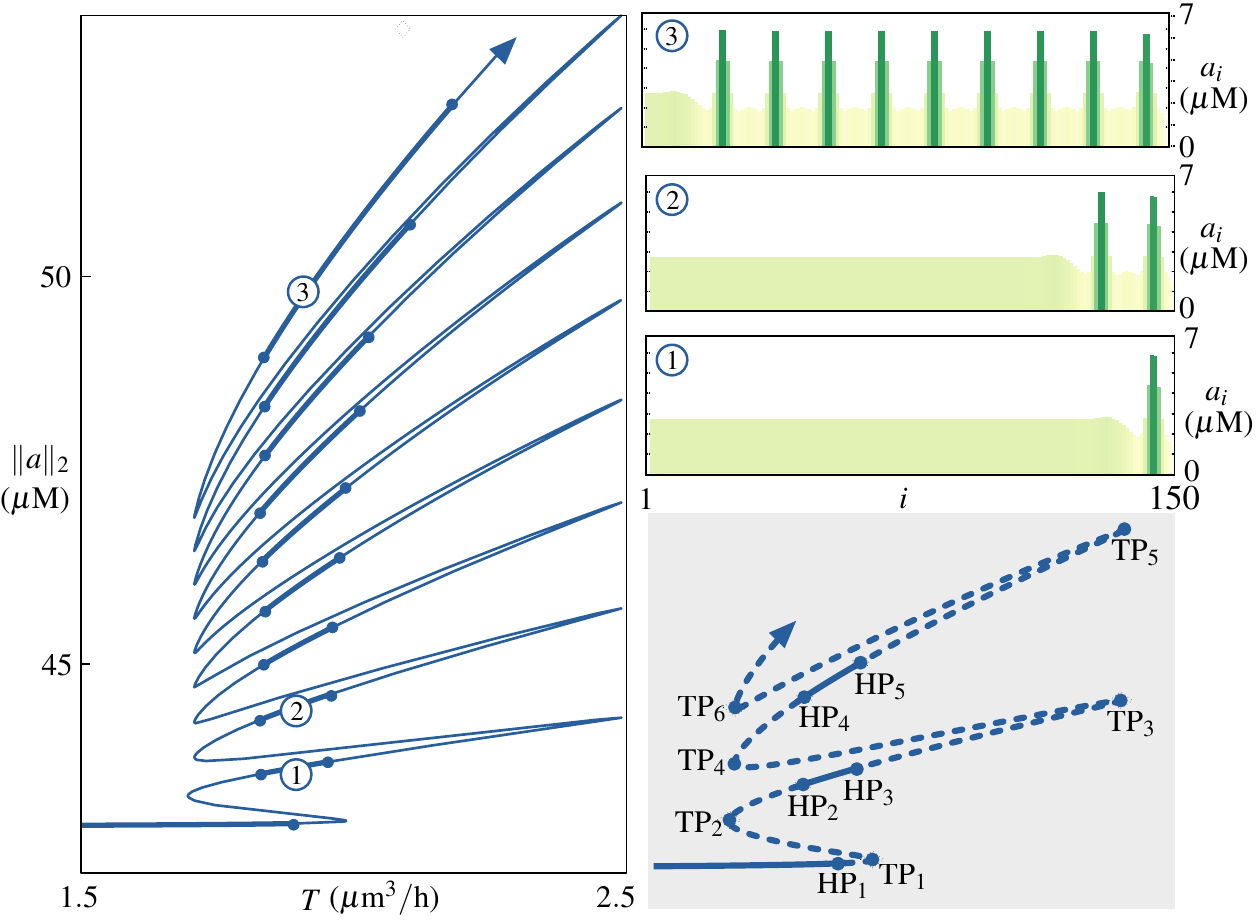}
  \caption{
    Bifurcation diagram and selected solution patterns for the Smith
    model posed on a row of $150$ identical cells with Neumann boundary
    conditions at $i=1$ and free boundary conditions at $i=n$ (see
    Figure~\ref{fig:FiniteDomainWithDiffusion}). Left: 2-norm of
    auxin concentration versus active transport parameter $T$. Right: as the
    snaking bifurcation diagram is ascended, new peaks are formed from the
    boundary towards the interior. Bottom-right: stable segments of the branch
    are found between Hopf bifurcations $\textrm{HP}_2$ and $\textrm{HP}_3$,
    $\textrm{HP}_4$ and $\textrm{HP}_5$, etc. Other secondary instabilities (not
    shown) are present along the unstable branches. Parameters: $D= 1\,\mu
    \textrm{m}^2/\textrm{h}$, $\rho_\IAA = 0.85 \, \mu\textrm{M}/\textrm{h}$;
    other parameters as in Table~\ref{table:parameters}.
  }
  \label{fig:bifdiagramr0850}
\end{figure}

In Figure~\ref{fig:bifdiagramr0850} we show a branch of solutions of the
Smith model for the 1D regular domain obtained with the parameter
set in Table~\ref{table:parameters}. We start the computation from the
homogeneous solution at $T=0$ and follow the pattern for increasing values of $T$. As
$T$ changes, we plot the $2$-norm of the auxin vector, $\Vert \va \Vert_2$,
which is a measure of the spatial extent of the solution (the lower $\Vert \va
\Vert_2$, the more localised the pattern) and denote stable (unstable) branches with
solid (dashed) or thick (thin) lines.

The low peak found close to the boundary persists for increasing values of $T$ and grows
steadily until we meet a first turning point (TP1). Before exploring the
diagram in full, we compare our numerical findings with the analytical
predictions of the asymptotic theory, valid for small $T$. The analytic
asymptotic profile~\eqref{eq:asymptoticSolutions} gives
a relative error $\Vert \va - (\va^* + T\boldsymbol{\alpha})
\Vert_2 / \Vert \va \Vert_2$ less than $0.4\%$ for $T \leq 0.2 \; \mu
\textrm{m}^3/\textrm{h}$, after which higher-order terms become predominant. This is
shown in Figure~\ref{fig:analyticVsNumericalBD} where we compare a branch of
approximate solutions (magenta) to a branch of solutions to the full nonlinear
problem (blue) for $D=1\, \mu \textrm{m}^2/\textrm{h}$ and small values of $T$. 

\begin{figure}
  \centering
  \includegraphics{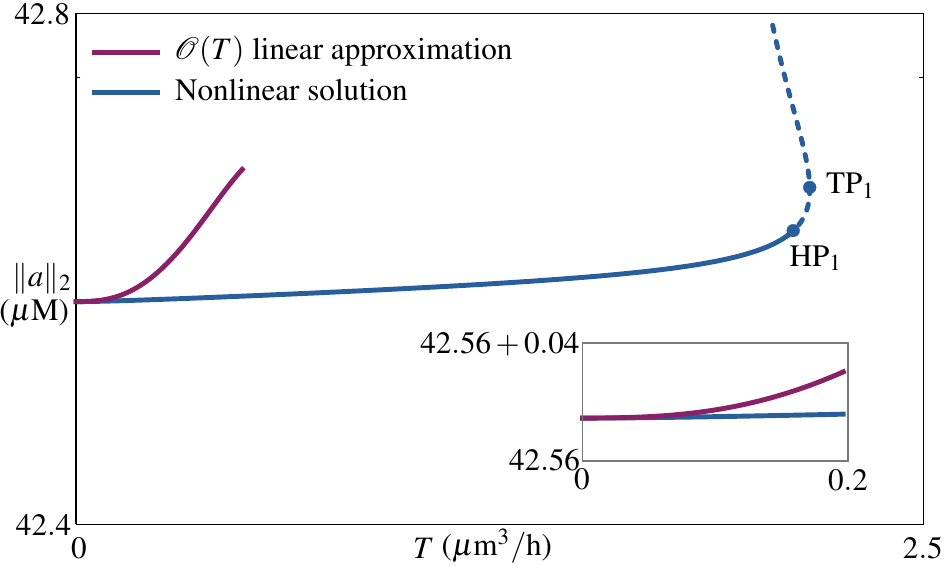}
  \caption{
    Comparison between a branch of approximate linear $\mathcal{O}(T)$
    solutions (magenta) and the corresponding branch of solutions to the full
    2-component nonlinear system by Smith (blue) for a row of 150 identical
    cells with homogeneous Neumann boundary conditions at $i=0$ and free boundary
    conditions at $i=n$. The plot shows the 2-norm of the IAA concentration
    vector versus the continuation parameter $T$. Parameters as in
    Figure~\ref{fig:bifdiagramr0850}.
  }
  \label{fig:analyticVsNumericalBD}
\end{figure}

As we ascend the bifurcation diagram in Figure~\ref{fig:bifdiagramr0850}, new peaks
are formed on the left side
of the existing peaks, that is, towards the interior of the domain, until the whole
domain is filled with peaks. We note that peaks are regularly spaced, as it
was also observed in \cite{Smith2006, Jonsson2006,  Reinhardt2003}. 

This bifurcation diagram resembles the one found for reaction--diffusion PDEs posed
on the real line~\cite{burke-knobloch:07,burke-knobloch:07c,beck-knobloch-etal:09}
except that here peaks are formed at the boundary rather than at the core of the
domain. When peaks fill the entire domain, the branch enters an unstable irregular
regime without snaking (not shown). Branches of solutions with peaks covering the
entire domain are also present (not shown) and are partially discussed in
Section~\ref{subsec:InfluenceOfSecondaryParameters}.

\begin{rem}[Biological interpretation of snaking] The diagram in
  Figure~\ref{fig:bifdiagramr0850} makes a plausible biological prediction for the
  formation of large peaks. A tissue composed of a string of cells, in the presence of
  passive diffusion, selects auxin patterns depending on the value of the active
  transport. Our analysis of the Smith model predicts that there exist two main
  regimes: for small $T$ there is a single low auxin peak at the boundary. As
  $T$ becomes larger, we enter a regime where the tissue can select from a large
  variety of auxin patterns. If, for instance $T \approx 1.9 \mu
  \textrm{m}^3/h$, the tissue is able to select pattern 1, 2 or 3, which have a variable
  number of peaks. Therefore, the pattern selected in experiments depends highly
  upon the initial conditions of the system, similarly to what was reported by
  J\"onsson and Krupinski~\cite{Jonsson2010}. 
  As we shall see, a slanted version of the snaking bifurcation diagram is also
  present in 2D domains (both regular and irregular and for both the Smith and
  the Chitwood models).
  \end{rem}

\begin{figure}
  \centering
  \includegraphics{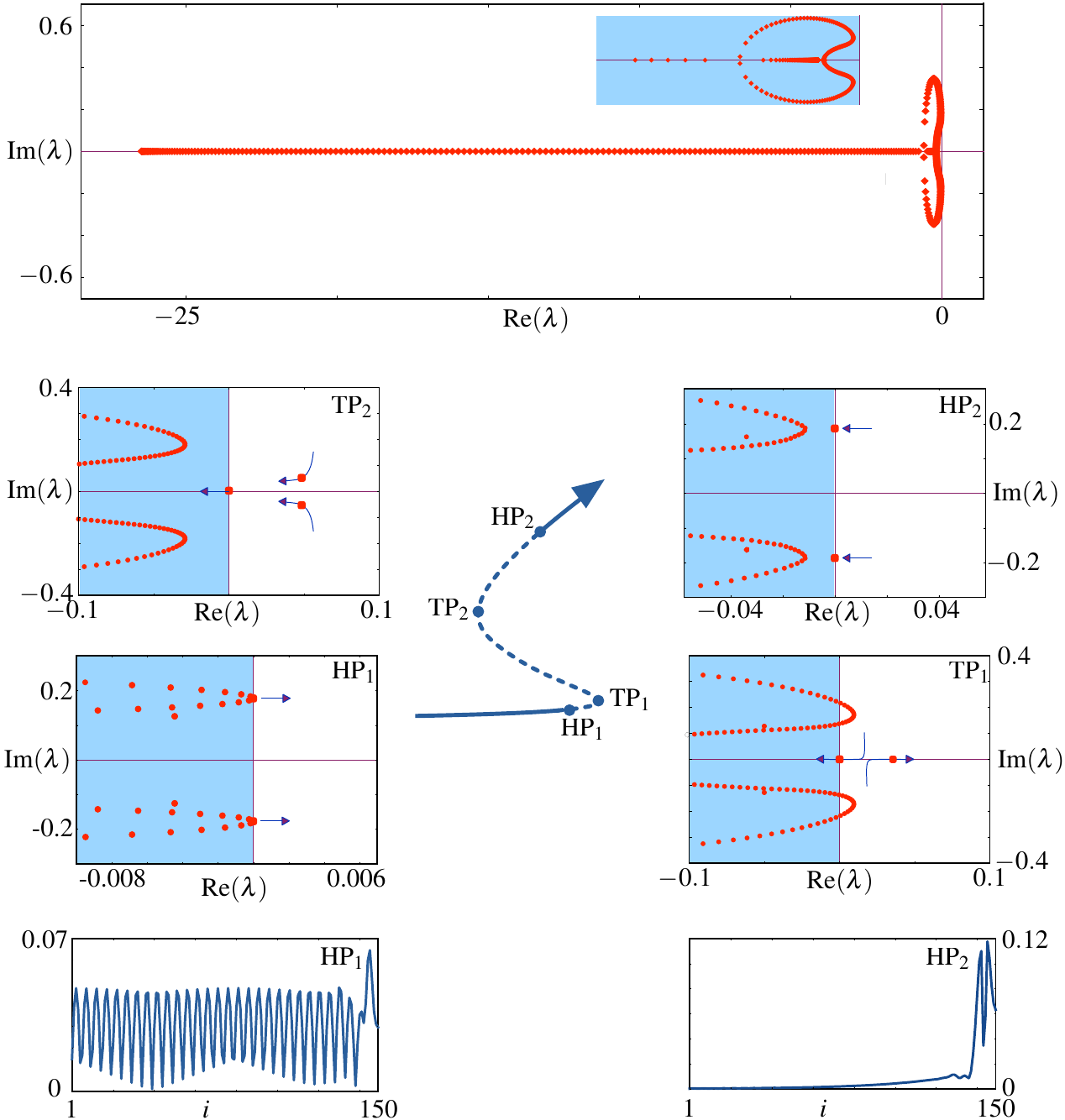}
  \caption{
    Spectral computations of selected solution profiles. Top panel: spectrum of a
    stable solution point for $T< T_{\textrm{HP}_1}$. Middle panels: spectra at
    selected bifurcation points along the snaking bifurcation branch of
    Figure~\ref{fig:bifdiagramr0850}. Bottom panels: unstable eigenfunctions at
    $\textrm{HP}_1$ and $\textrm{HP}_2$. Parameters as in
    Figure~\ref{fig:bifdiagramr0850}.
    }
    \label{fig:spectrum}
\end{figure}

\subsubsection{Instabilities on the snaking branch}
To understand if a solution pattern in stable to small
  perturbations we study in detail the eigenvalues of the Jacobian
  matrix (Result \ref{res:snaking}).  The Jacobian matrix for the spatially-extended
system~\eqref{eq:SmithModel1}--\eqref{eq:SmithModel2} is sparse with a
characteristic block structure determined by transport and diffusion
terms (we refer the reader to~\cite{Draelants2013} for a detailed
description) and, for relatively small systems such as this one,
eigenvalues are computed with dense linear algebra routines.

In this example, the solution with one small peak 
at the boundary becomes unstable at a Hopf bifurcation ($\textrm{HP}_1$) at
$T\approx 2.1  \; \mu \textrm{m}^3/\textrm{h}$, closely
followed by other oscillatory instabilities and a saddle-node bifurcation
($\textrm{TP}_1$) at $T\approx 2.2\; \mu \textrm{m}^3/\textrm{h}$, after
which the solution remains unstable. On the snaking branch, we find that saddle-node
and Hopf bifurcations alternate regularly, as documented in
Figure~\ref{fig:bifdiagramr0850}: saddle-node bifurcations align at 
$T\approx 1.9 \; \mu \textrm{m}^3/\textrm{h}$ and $T \approx 2.5
\; \mu \textrm{m}^3/\textrm{h}$, while Hopf bifurcations depart from each other as
patterns become less localised. In this parameter setting, stable portions
of the branch are delimited by two Hopf bifurcations, which, to the best of our
knowledge, has not been reported before for snaking in reaction--diffusion systems.
It should be noted, however, that Burke and Dawes~\cite{burke2012localized} found
Hopf bifurcations at the bottom of the snaking branch for an extended
Swift--Hohenberg equation, which may lead to a bifurcation structure similar to the
one in Figure~\ref{fig:bifdiagramr0850} if secondary parameters are varied.

In Figure~\ref{fig:spectrum} we show spectra of solutions at selected
points on the branch. Overall these spectra resemble those found in
discretised advection-diffusion PDEs, with largely negative real eigenvalues
associated with diffusion terms of the governing equations. In this context,
however, increasing the number of cells does not alter the cell spacing, hence
the spectrum does not grow in the negative real direction for larger system
sizes. 
\begin{figure}
\centering
\includegraphics{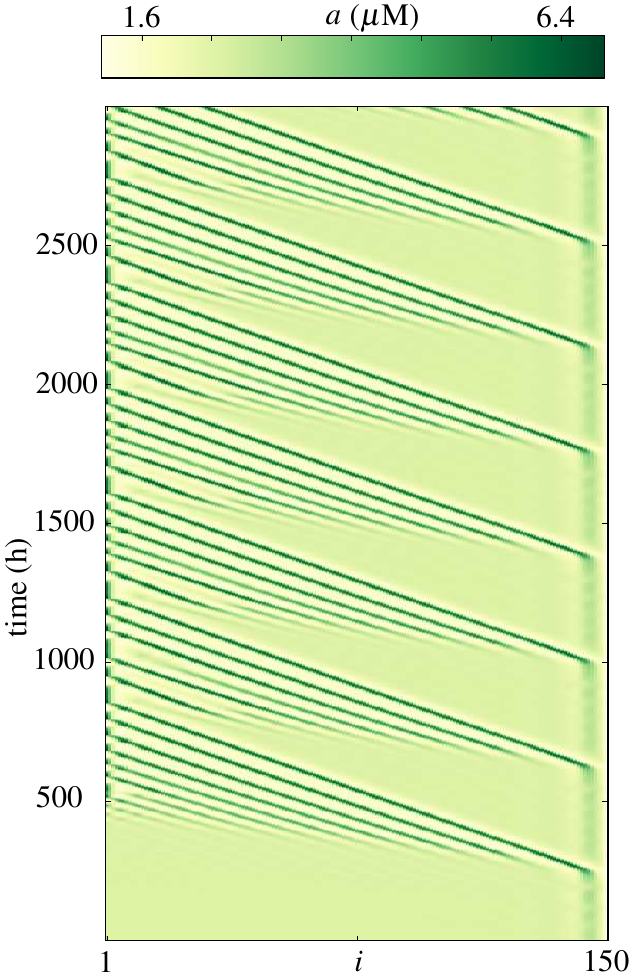}
\caption{
  Spatio-temporal state obtained via time simulation of the Smith model
  posed on a one-dimensional domain, close to the Hopf bifurcation $\textrm{HP}_1$. 
  We set $T \gtrsim T_{HP_1}$ (and other parameters as in
  Figure~\ref{fig:bifdiagramr0850}) and use as initial condition a steady state
  with one peak at the boundary, obtained for $T \lesssim T_{HP_1}$. A long-time
  periodic auxin wave travels (and new auxin peaks are recruited) from the
  boundary towards the interior of the domain.
}
\label{fig:timeEvolution}
\end{figure}

We monitored spectra of localised solutions as the snaking branch was ascended
(see Figure~\ref{fig:spectrum}): immediately after the Hopf bifurcation HP1,
multiple eigenvalues cross the imaginary
axis, therefore several oscillatory instabilities exist between $\textrm{HP}_1$
and $\textrm{TP}_1$ ($\textrm{HP}_2$ and $\textrm{TP}_2$, etc.). In the bottom
panel of Figure~\ref{fig:spectrum} we show that the Hopf
eigenfunction at $\textrm{HP}_1$ has a maximum near the boundary
and the one for $\textrm{HP}_2$ is also spatially localised. 
We expect that branches of time-periodic (possibly spatially-localised)
solutions emerge from the Hopf bifurcations. We have not observed stable
small-amplitude oscillations in direct numerical simulations, but we report the
existence of stable periodic states in which a temporal oscillation of the peak
at $i=n$ initiates a wave of auxin moving towards the boundary at $i=1$, with
long oscillation periods.  

In Fig.~\ref{fig:timeEvolution} we show such a periodic solution obtained via time
simulation in the neighbourhood of $\textrm{HP}_1$ (which is also visible in
Figure~\ref{fig:analyticVsNumericalBD}). We set $T
\gtrsim T_{\textrm{HP}_1}$ and use as initial condition a steady state (with one
peak at the boundary), obtained for $T \lesssim T_{\textrm{HP}_1}$. In the
resulting periodic state, auxin
peaks are dynamically formed from the tip towards the interior of the domain: we
point out that the period of oscillations (about 377 hours) is
much greater than the period associated to the unstable Hopf eigenvalues. In addition, on
such long time scales it is reasonable to assume that new cells are formed, so the
geometry of the problem should change as well.

It was recently shown by Farcot and Yuan that, in one-dimensional flux-based
models with no-flux boundary conditions, active transport is sufficient to
elicit auxin oscillations~\cite{Farcot2013}. In the concentration-based model
considered here, oscillatory states in regular one-dimensional arrays are also
found in a regime where active transport dominates over diffusion.

\subsubsection{Changes in the auxin production parameter}
\label{subsec:InfluenceOfSecondaryParameters}

\begin{figure}
  \centering
  \includegraphics{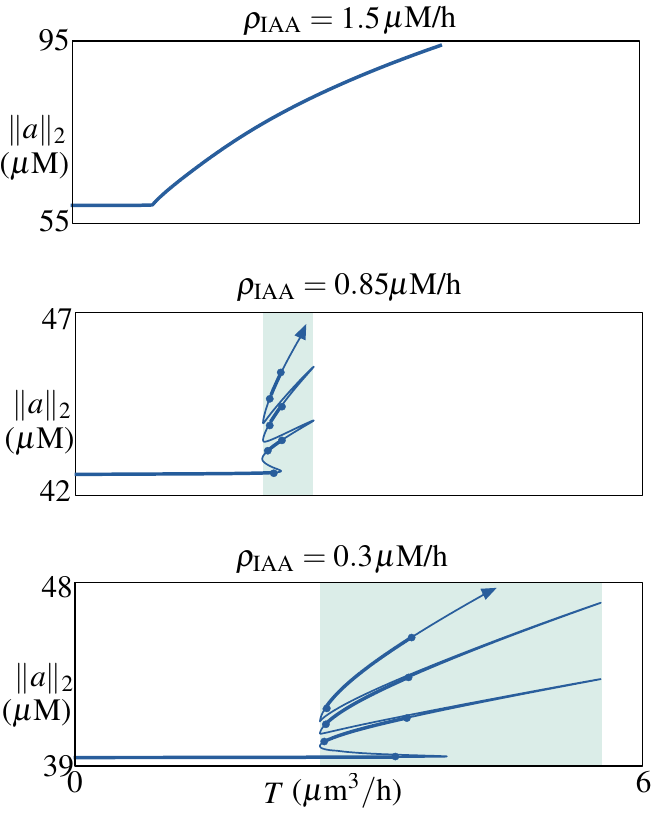}
  \caption{
    Bifurcation diagrams of the Smith model for various values of
    the auxin production coefficient $\rho_\IAA$ (other parameters as in
    Figure~\ref{fig:bifdiagramr0850}). The lower the $\rho_\IAA$, the wider the
    snaking width and the stable branches. For sufficiently large values of
    $\rho_\IAA$ the snaking disappears and peaks form at once on a fully stable
    branch.
  }
  \label{fig:bifdiagramMulti}
\end{figure}
We conclude this primer on the 1D tissue by investigating the robustness of the
snaking scenario described above (Results \ref{res:robustness}). In~\cite{Draelants2013} it was shown that the
auxin production parameter $\rho_\IAA$ has a significant influence on the solution
profiles, therefore it
is interesting to study how changes in this parameter affect the bifurcation
structure. We repeated the numerical continuation of the Smith model for $20$
values of $\rho_\IAA$ in the interval $[0.3 \; \mu \textrm{M}/\textrm{h},1.5 \;
\mu \textrm{M}/\textrm{h}]$. For low values of $\rho_{_{\IAA}}$, both the
oscillatory instability $\textrm{HP}_1$ and the saddle node $\textrm{TP}_1$ move
to the right and give rise to snaking bifurcation diagrams with increasingly
wider stable segments (see Figure~\ref{fig:bifdiagramMulti}). As a
consequence, in biological experiments where the auxin production was kept to a
low value, the tissue would support multiple auxin patterns for a wider range of
active transport coefficients.
In the limit $\rho_{_{\IAA}} \ll \mu_{_{\IAA}}$
decay dominates over production, hence large peaks can not be sustained and
indeed we find that the solution with a single small peak at the tip persists
for very large values of $T$.
This is in line with previous papers by Sahlin \etal \cite{Sahlin2009}
and De Reuille \etal \cite{deReuille2006} where it was postulated that auxin
patterning demands a minimal level of auxin production within the tissue.

On the other hand, increasing $\rho_\IAA$ causes the snaking diagram
to shrink and then disappear for $\rho_\IAA \geq 1.2\; \mu \textrm{M}/\textrm{h}$. In
Figure~\ref{fig:bifdiagramMulti} we show a fully stable branch for $\rho_\IAA = 1.5\; \mu
\textrm{M}/\textrm{h}$. On this branch peaks
develop at once from the small-amplitude solution, without turning points. We
mention however that for $\rho_\IAA$ between $1.2\; \mu \textrm{M}/\textrm{h}$ and $1.3\; \mu
\textrm{M}/\textrm{h}$, Hopf bifurcations are found along the non-snaking branch (not
shown), similar to what is found for the infinite domain~\cite{Draelants2013}.

As snaking branches distort, several types of secondary instabilities
and collisions with neighbouring branches occur. In particular, we
report codimension-2 Bogdanov--Takens bifurcations originating from
the collision between $\textrm{TP}_2$ and $\textrm{HP}_2$, ($\textrm{TP}_4$ and
$\textrm{HP}_4$, $\textrm{TP}_6$ and $\textrm{HP}_6$, etc.)
when $\rho_\IAA$ is varied. The existence of these codimension-2
bifurcations could also be envisaged from the spectra in
Figure~\ref{fig:spectrum}. 
These instabilities, as the ones reported in
the previous section, indicate that the tissue is capable of sustaining
oscillations and dynamical auxin patterning, as well as steady states with
multiple peaks. Dynamic states with spatio-temporal coherence (such as the one
reported in Figure~~\ref{fig:timeEvolution}) are interesting from a biological
standpoint~\cite{Merks2007}, as they occur for the biologically plausible
parameter values reported in Table~\ref{table:parameters}. However, we could not
find them in 2D domains with realistic parameter values, hence we do not discuss
them further in this paper.
\label{subsec:TwoDimensionalDomain}

\begin{figure}
  \centering
  	\includegraphics{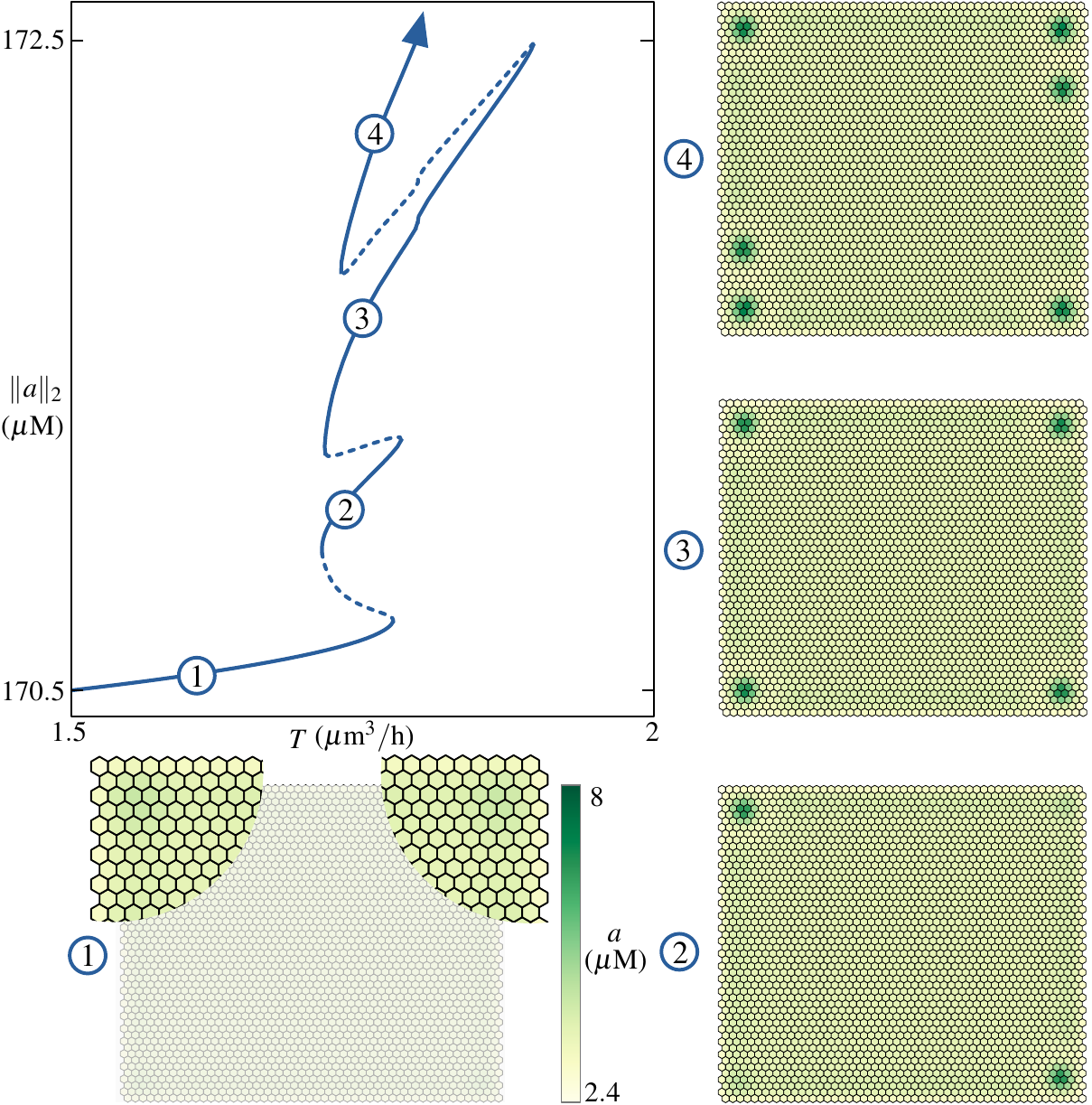}
  \caption{
  Bifurcation diagram and selected solution patterns for the Smith
  model posed on a regular array of $50$ by $50$ hexagonal cells with 6
  neighbours in the interior, 3 neighbours at the left and right edges and 4
  neighbours at the top and bottom edges. Larger peaks are developed initially at
  the top-left and bottom-right corner (see also the values of the pre-factors
  $\xi_i$ in Figure~\ref{fig:xiHexagons}) and new peaks are formed along the left and
  right edges, where we have fewer neighbours. 
  Parameters: $D= 1\,\mu \textrm{m}^2/\textrm{h}$, $\rho_\IAA = 1.5 \,
  \mu\textrm{M}/\textrm{h}$; other parameters as in~Table \ref{table:parameters}. 
  }
  \label{fig:regular2Dsnaking}
\end{figure}

\subsection{Two-dimensional domains}
\label{subsec:2DDom}
We now move to more realistic geometries and study 2D domains with approximately
square and circular boundaries, on which we prescribe free boundary conditions.
In this Section we will revisit
Results~\ref{res:HomState}--\ref{res:irregular} in the 2D setting for both the
Smith and the Chitwood  models, so we refer the reader to the general summary in
Section~\ref{sec:resultsSummary} and the 1D primer in
Section~\ref{subsec:TheExistenceOf}.

The methods described in the previous section apply straightforwardly to the 2D
case. In the first example we consider the Smith model on a grid of $50$
by $50$ hexagonal prismic cells with $l_{ij} = 1 \, \mu\textrm{m}$ and $V_i =
3\sqrt{3}/2 \, \mu\textrm{m}^3$. Cells have $6$ neighbours in the
interior, $3$ neighbours at the left and right edges, and $4$ neighbours at the
top and bottom edges.
In this domain, corners are not all equal (see Figure~\ref{fig:regular2Dsnaking})
and we chose this configuration intentionally, to illustrate the influence of
the number of neighbouring cells on the emerging patterns in 2D domains. 
\begin{figure}
  \centering
  \includegraphics{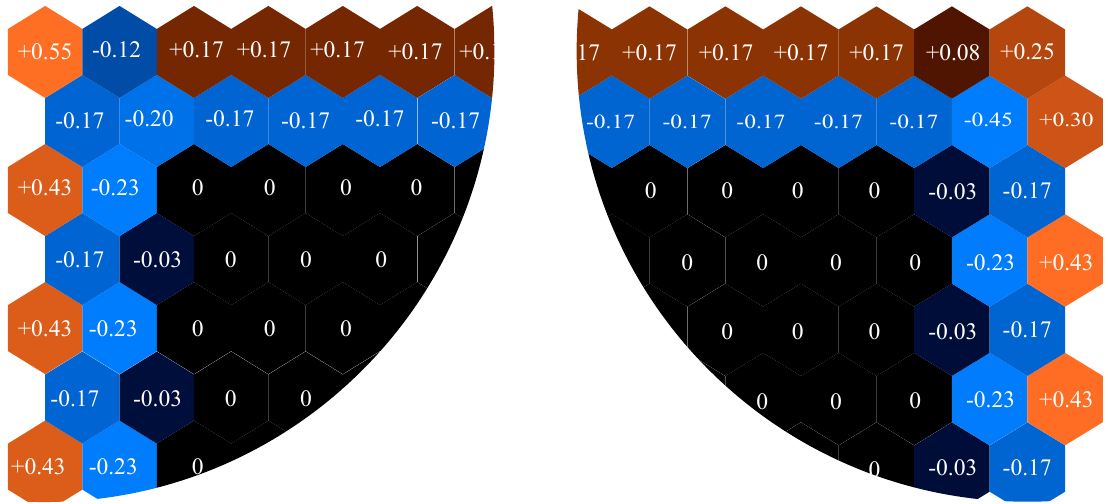}
  \caption{
    Geometric pre-factors $\vxi$ at the top-left and top-right corners of 
    the 2D regular domain of Figure~\ref{fig:regular2Dsnaking}. The interior set
    $\mathcal{I}$ is now clearly visible. Peaks and dips are proportional, for
    small $T$ and $D=0$, to the values of $\xi_i$. When $D \neq 0$ the largest
    peak is formed in the interior, close to the top/left corner and a smaller
    one is formed in the interior, towards the top/right corner.
  }
  \label{fig:xiHexagons}
\end{figure}

Figure~\ref{fig:xiHexagons}
shows the values of the geometric pre-factors $\boldsymbol{\xi}$ for 2 corners of
the domain: our asymptotic analysis for $D=0 \; \mu \textrm{m}^2/\textrm{h}$
(see the discussion in
Section~\ref{subsubsec:AFiniteOneDimensionalDomainWithoutDiffusion} and the
generic derivation in Section~\ref{sec:MaterialsAndMethods})
predicts the formation of peaks at the boundaries, with the highest peak at the
top-left and bottom-right corners. 
Numerical computations for positive $D$ show that these peaks persist and become
prominent for increasing $T$ (Figure \ref{fig:regular2Dsnaking}, pattern 2). As
in the one-dimensional case, patterns are arranged on a snaking bifurcation branch,
even though in two-dimensions the snaking is slanted. Peaks arise initially in
all four corners, then new spots are formed along the left and right edges
(where cells have fewer neighbours), and then, for sufficiently large values of
$\rho_\IAA$, along the top and bottom edges. In contrast with the 1D case, we
have not found oscillatory bifurcations in this region of parameter
space, so we conclude that stable portions of the branch are now delimited by
turning points (see Figure~\ref{fig:regular2Dsnaking}). From a
biological perspective, this means that patterns with peaks at the boundary are
more likely to be observable in experiments, as they are stable in a wider region of
parameter space.

\begin{figure}
  \centering
  	\includegraphics{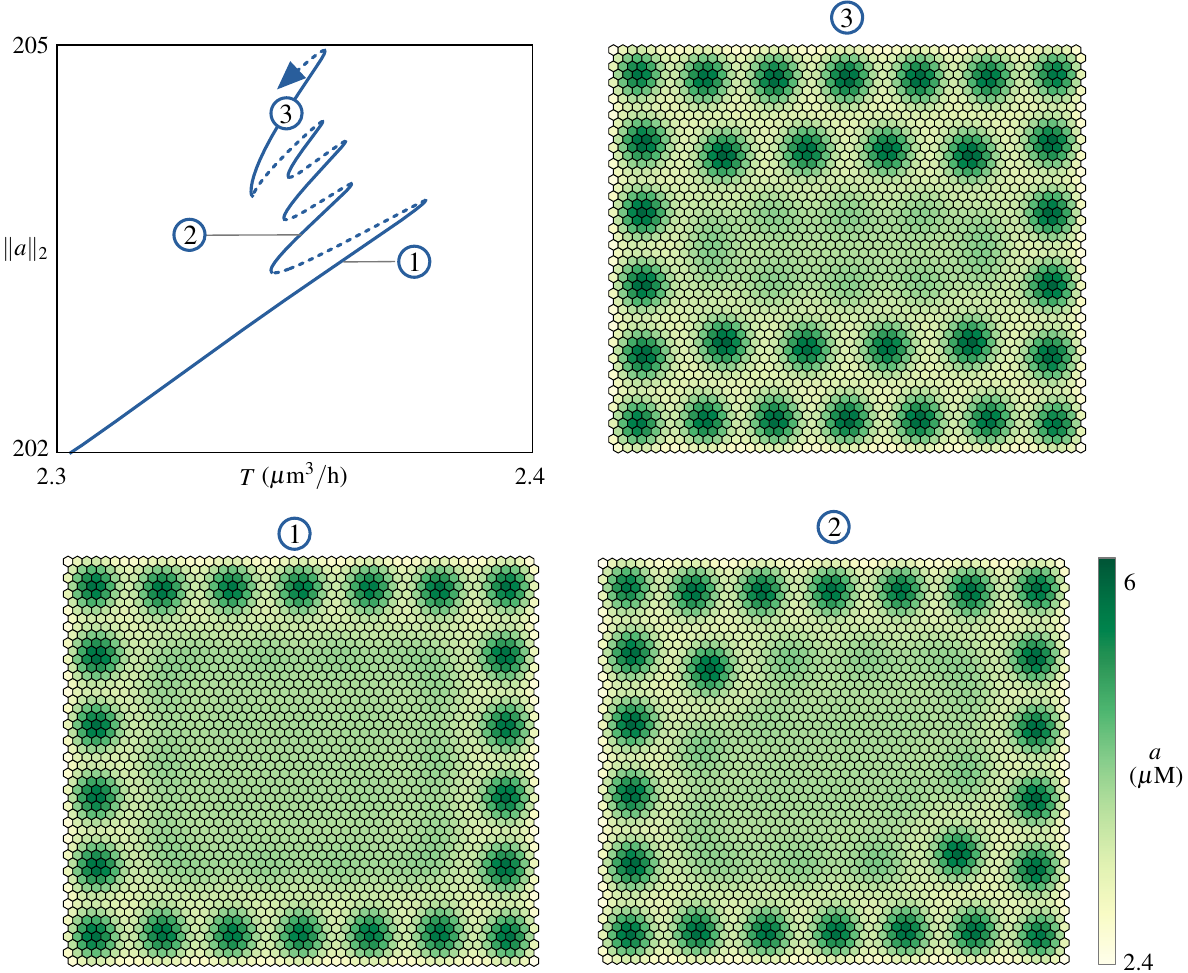}
  \caption{
  Bifurcation diagram and selected solution patterns for the Chitwood
  model posed on the same domain of Figure~\ref{fig:regular2Dsnaking} and a
  slightly larger value of the auxin production coefficient ($\rho_{_{\IAA}}=2
  \mu \textrm{M}/\textrm{h}$, $D= 1\,\mu \textrm{m}^2/\textrm{h}$ and other
  parameters are as in~Table \ref{table:parameters}).
  A ring of peaks is developed at the boundary, owing to the increased value of
  $\rho_\IAA$. Other spots are formed in pairs (pattern 2), until they fill a
  full row (pattern 3) and the whole domain (not shown).
  }
  \label{fig:regular2DsnakingChitwood}
\end{figure}

In a second example, we consider the Chitwood model posed on the
2D regular domain, using the parameters of Table~\ref{table:parameters}.
Remarkably, the resulting bifurcation diagram (not shown) is analogous to the
one found in Figure~\ref{fig:regular2Dsnaking} for the Smith model. As a
further confirmation, we tested the robustness of the snaking mechanism to
changes in the auxin production coefficient, as it was done for the 1D case in
Figure~\ref{fig:bifdiagramMulti}.  We set $\rho_\IAA=2\, \mu
\textrm{M}/\textrm{h}$ and show in Figure~\ref{fig:regular2DsnakingChitwood} the
corresponding bifurcation diagram and solution patterns. While active transport
remains responsible for the selection of peaks towards the boundary, the
interplay with auxin production allows the formation of a ring of spots at the
boundary as opposed to single spots at the corners (see pattern 1 in
Figure~\ref{fig:regular2DsnakingChitwood}). After the
first turning point, spots are formed in pairs (pattern 2), until they
fill a full row (pattern 3) and the whole domain (not shown). As in Section
\ref{subsec:InfluenceOfSecondaryParameters}, the auxin production coefficient has a
large influence on the resulting peaks.  For this 2D regular domain, we also scanned
several values of the auxin production coefficient and confirmed that comparatively
high values of $\rho_\IAA$ induce the formation of peaks all-at-once (similar to top
panel of Figure~\ref{fig:bifdiagramMulti}), hence a fully patterned tissue is
possible without a Turing bifurcation.

\begin{figure}
  \centering
  \includegraphics{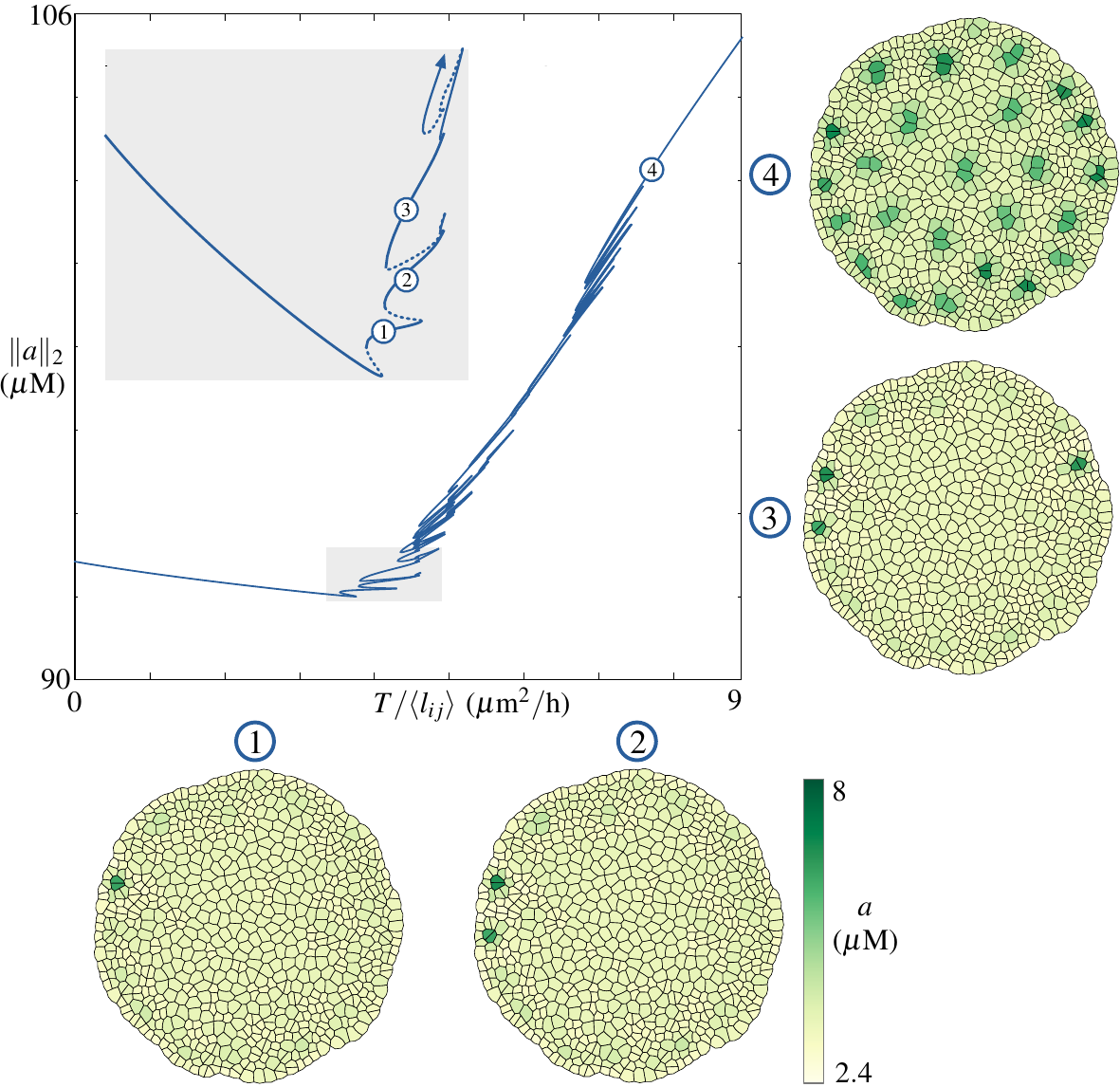}
  \caption{
    Bifurcation diagram and selected solution patterns for the Smith model
    for an almost-circular domain of $742$ irregular cells (geometry taken
    from~\cite{Merks2011}). We find an irregular and slanted bifurcation diagram
    with stable portion delimited by saddle-node bifurcations. 
    Parameters: $D= 1\,\mu \textrm{m}^2/\textrm{h}$, $\rho_\IAA = 1.5 \,
    \mu\textrm{M}/\textrm{h}$; other parameters as in
    Table~\ref{table:parameters}.
  }
  \label{fig:irregularsnaking}
\end{figure}

\begin{figure}
  \centering
  \includegraphics{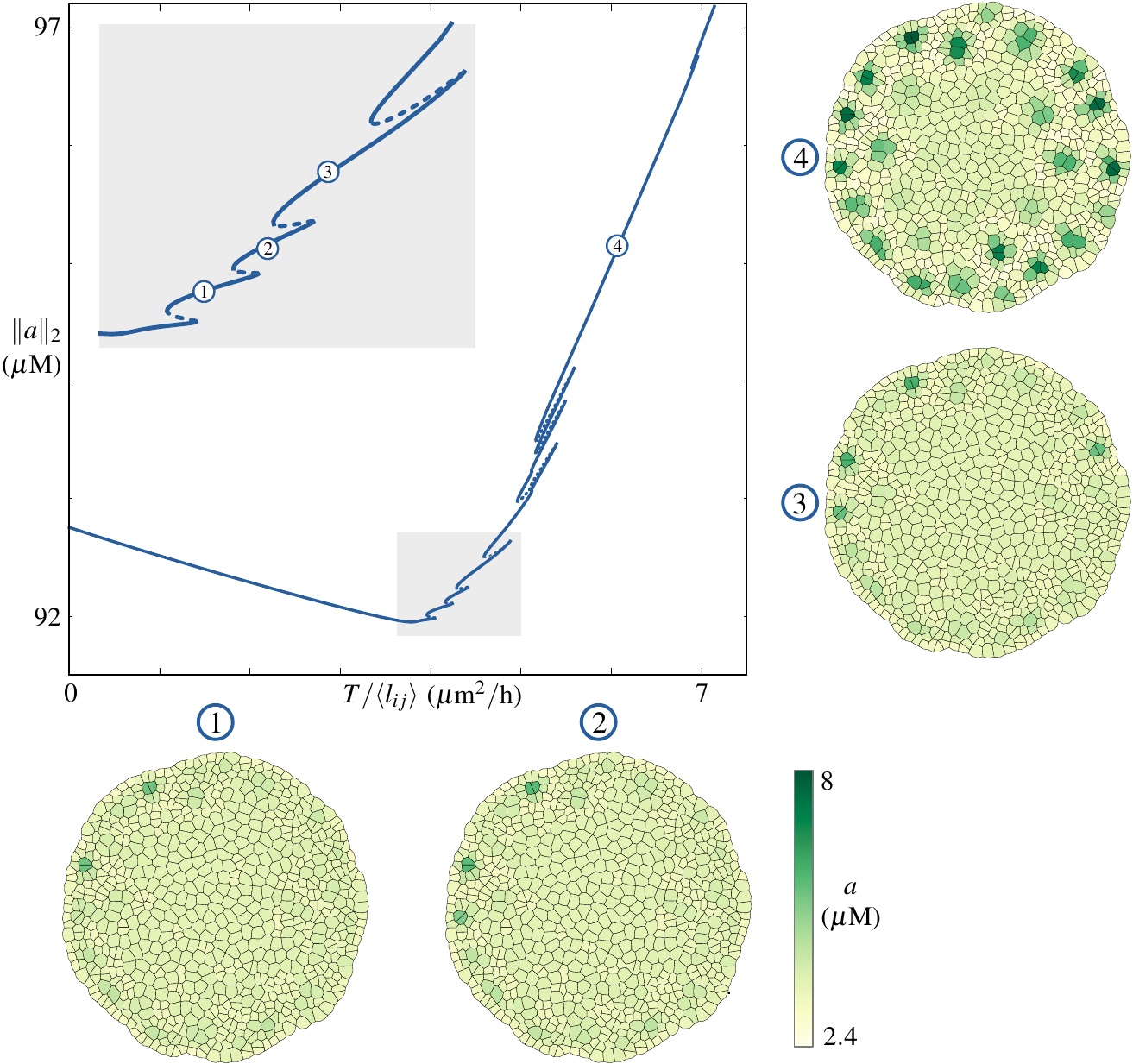}
  \caption{
  Bifurcation diagram and selected solution patterns for the Chitwood model for
  an almost-circular domain of $742$ irregular cells (geometry taken
  from~\cite{Merks2011}). We find an irregular and slanted bifurcation diagram with
  stable portion delimited by saddle-node bifurcations. Parameters: $D= 1\,\mu
  \textrm{m}^2/\textrm{h}$, $\rho_\IAA = 1.5 \, \mu\textrm{M}/\textrm{h}$; other
  parameters as in Table~\ref{table:parameters}.
  }
  \label{fig:irregularsnakingChitwood}
\end{figure}

In the remaining 2 examples, $742$ irregular prismic cells cover an almost-circular
domain in a realistic tissue  with free boundary conditions (see Figure
\ref{fig:irregularsnaking}, whose geometry has been extracted
from~\cite{Merks2011}). Even though the asymptotic analysis is not valid for irregular arrays, we
expect results to be qualitatively similar if cellular volumes and
contact areas do not vary greatly from cell to cell. In these examples
the number of neighbours varies over the domain; however, the cells at
the boundaries have predominantly fewer neighbours and this is where
peaks are formed initially. The bifurcation diagram is now plotted in terms of the
scaled bifurcation parameter $T/\langle l_{ij} \rangle$, where $\langle \cdot
\rangle$ denotes the average in the tissue (so far we have considered cases
where $\langle l_{ij}\rangle=1\,\mu\textrm{m}$, so the following diagrams are
directly comparable with the ones above). We note that in this domain $\langle
l_{ij} \rangle \approx 13.26 \, \mu\textrm{m}$ and $\langle V_i
\rangle \approx 294 \, \mu\textrm{m}^3$.  
As in the 2D regular cases,
stable portions of the branch are enclosed between consecutive
saddle-nodes bifurcations and there are no oscillatory instabilities on
the stable branches.  
Figure~\ref{fig:irregularsnaking} shows the results for the Smith model:
as usual, peaks are formed initially at the boundary and then fill the interior
(see pattern 4), and the slanted snaking ensures the existence of stable
solutions with localised peaks in a wide regime of the parameter $T$.

When we pose Chitwood model on the same irregular domain, the results
are strikingly similar, as seen in Figure~\ref{fig:irregularsnakingChitwood}:
peaks form initially at the top-left quadrant of the circular domain (see
patterns 1 and 2 in Figures~\ref{fig:irregularsnaking} and
\ref{fig:irregularsnakingChitwood}), confirming that it is the geometry of the
tissue to drive the spots location.  An inspection of the fully patterned
tissues (patterns 4 in Figures~\ref{fig:irregularsnaking} and
\ref{fig:irregularsnakingChitwood}) reveal that model parameters and functional
forms for the active transport functions have an influence on the size and structure of the peaks. Variations in $\rho_{_{\IAA}}$ also
confirmed the trend seen in Figure~\ref{fig:bifdiagramMulti} (not shown).

\section{Discussion}
\label{sec:discussion}
In this paper we investigated the origin of auxin peaks in generic
concentration-based models and proposed a robust mechanism for their formation
over short time scales, using a combination of asymptotic and numerical
bifurcation analysis.

The asymptotic calculations, valid for a class of models with identical cells
and weak active transport, show that peaks emerge as boundary corrections to the
homogeneous steady state: the peak amplitude depends on the
local geometry and is higher in regions where cells have fewer neighbours, that
is, next to the boundary. 
Crucially, this is a direct consequence of the mathematical structure of the
models considered here (Hypotheses \ref{hyp:domain}--\ref{hyp:nu}):
since 
the active transport depends on the number of neighbours at distance $2$ from
the $i$th cell via the geometric coefficients $\xi_i$, then deviations from the
homogeneous state will always appear at the boundaries, where the number of
neighbours is different from the interior of the tissue. This mechanism
is different from (and not in contrast with) the Turing bifurcation
scenario reported in previous studies on unbounded domains: on finite tissues,
peaks do not emerge from instabilities of the flat state, but they simply morph
from it for low values of $T$. The most immediate consequence of our
mathematical analysis is that, in concentration-based models, active transport
and geometry concur to promote localisation of auxin peaks~\cite{Heisler2005}.

In irregular domains, a similar asymptotic analysis can be carried out, but the
peak selection mechanism in this case also depends on the cellular volumes and
contact areas, so we can not exclude a priori that peaks will form in the
interior as well as on the boundary. A statistical characterisation of the peaks
location in relation to the variance of the cellular array is possible and
should be considered in future studies.

The two models by Smith~\cite{Smith2006} and Chitwood~\cite{Chitwood2012}
fit in the framework discussed above and, for these systems, we have
provided numerical evidence that peaks persist for moderate and large values of
the active transport rate $T$: large-amplitude peaks are arranged on a snaking
branch, which becomes slanted in 2D tissues.
A major implication of the numerical findings of
Figures~\ref{fig:irregularsnaking}--\ref{fig:irregularsnakingChitwood}
is that localised auxin peaks are observable: a slow experimental sweep in
the active transport from low to high values should reveal an initial
localisation of the peaks, which then progressively de-localise and fill the
domain. Conversely, if the active transport is kept constant, the tissue should
be able to select from more than one pattern, depending on the initial
condition.

Since snaking is now recognised as the footprint of localisation in a
wide variety of nonlinear media, we expect that bifurcation diagrams similar to
the ones shown in this paper for the Smith and the Chitwood models might also
arise generically. One class of models that are of interest and can be analysed
via numerical bifurcation analysis are flux-based auxin models, which have not
been studied in this paper. While numerical continuation is
readily applicable to such models, they are likely to require a separate
analytic treatment, as some of them do not possess the factorisation presented in
Hypothesis~\ref{hyp:nu}. 

Importantly, we find that the bifurcation scenario is influenced by the auxin
production rate, since the selectable configurations depend sensitively on the
balance between auxin production and active transport.
The results in Figure~\ref{fig:bifdiagramMulti} (also confirmed by the 2D
irregular calculations of Figures~\ref{fig:irregularsnaking} and
\ref{fig:irregularsnakingChitwood}) support the conclusion that if auxin
production rate was decreased quasi-statically, either actively or passively,
the organism would be able to switch from fully-patterned states to
configurations with few peaks at the boundary. In addition, since the parameters
of the concentration-based models considered here are scaled by cellular
volumes, we expect tissues with different cell sizes to behave similarly: in
tissues with larger $V$, the snaking limits are expected to occur for larger
values of $T$, so as to keep the ratio $T/V$ constant (and a similar reasoning
is valid for the passive transport parameter
$D$).

The conclusion reported above are naturally limited to experiments that are well
approximated by concentration-based models, and for the plausible biological
parameter values selected in the original papers by Smith \etal~\cite{Smith2006}
and Chitwood \etal~\cite{Chitwood2012}. We note that an experimental validation of
the predictions presented here requires the ability to detect changes in the
auxin distribution during development. An experimental technique
that could help testing the predictions of these two models is the one
recently proposed by Brunoud \etal~\cite{Brunoud2012}, which allows to visualise
auxin with high spatio-temporal resolution. We note that it would be possible to
apply numerical bifurcation analysis also to a modified model that
accounted for markers' dynamics and auxin-sensor interactions. 

A further desirable property of the experimental setup would be the
ability to stimulate auxin peaks locally (thereby changing initial conditions)
and test whether the tissue settles to a new equilibrium. In view of the large
uncertainty on the parameter values of the models, we expect our predictions to
agree qualitatively with experimental results.

\section{Materials and Methods}
\label{sec:MaterialsAndMethods}

In this section we present an analytical framework to construct steady
state solutions featuring localised auxin peaks in generic concentration-based
models. 

\subsection{Asymptotic derivation of peak solutions}
We begin by giving a generic definition of concentration-based models in the
absence of diffusion: as mentioned above, several examples from literature can
be cast in this form.
\begin{dfn}[Concentration-based model without diffusion]\label{def:conBasedModelNoDiff}
 A concentration-based model without diffusion is a set of $m \times n$ ODEs of the form
\begin{equation}
  \dot{\vy}_i = \vpi(\vy_i) - \vdelta(\vy_i) 
  + T \sum_{j \in \neigh_i} \vnu_{ji}(\vy_1,\ldots,\vy_n) - \vnu_{ij}(\vy_1,\ldots,\vy_n),
  \qquad i = 1,\ldots,n,
  \label{eq:concBasedModel}
\end{equation}
where $\vpi, \vdelta \colon \R_+^m \to \R_+^m $, are the production
and decay functions, respectively, $T \in \R_+$, is the (nonnegative) active
transport parameter, $\{1,\ldots,n\}$ are vertices of a static undirected graph
$G$, $\neigh_i \subseteq \{1,2,\ldots,n\}$ is the set of neighbours of cell $i$,
containing $\vert \neigh_i \vert$ elements and $\vnu_{ij}: \R_+^m \times \cdots
\times \R_+^m \to \R_+^m$
are the active transport functions. We will assume $\vpi$, $\vdelta$ and
$\vnu_{ij}$ to be smooth vector fields depending on a set of control parameters
$\vp \in \R^p_{+}$, but we omit this dependence for simplicity and write, for
instance, $\vpi(\vy_i)$ instead of $\vpi(\vy_i;\vp)$.
\end{dfn}
	\begin{rem} Concrete examples of concentration based models in this form can be found 
			in section \ref{sec:one-dimensional} an section \ref{sec:one-dimensional-2-comp}
	\end{rem}

We now prove the following result
\begin{lem}\label{lemma1}
  Let us consider the concentration-based model~\eqref{eq:concBasedModel} and
  let us suppose that there exist vector-valued functions $\vpsi \colon \R_+^m
  \times \R_+^m \to \R_+^m$ and $\vphi \colon \R_+^m \to \R_+^m $ such that
  \begin{equation}
  \vnu_{ij} (\vy_1,\ldots,\vy_n) = \vpsi(\vy_i,\vy_j) \odot \vphi(\vy_j)
  \oslash \sum_{k \in \neigh_i} \vphi (\vy_k), \qquad \text{for all $i,j =
  1,\ldots,n$,}
  \label{eq:factorization}
  \end{equation}
  where $\odot$ and $\oslash$ denote the standard Hadamard product and division
  between vectors. Further, let $\vys \in \R^m$ be such that $\vpi(\vys) =
  \vdelta(\vys)$, $\vpsi(\vys,\vys) \neq \mathbf{0}$ and $\vphi(\vys) \neq
  0$, then 
  \begin{enumerate}
    \item If $T=0$ or all cells have the same number of neighbours, $\vert
      \neigh_i \vert = \vert \neigh^* \vert$, then the
      homogeneous solution $(\vy^*,\ldots,\vy^*)^\textnormal{T} \in \R^{nm}$ is a steady state
      for the concentration-based model.
    \item If $0< T \ll 1$ and cells have different number of neighbours and the
      Jacobian matrix $\vpi'(\vy^*) - \vdelta'(\vy^*)$ is nonsingular, then a
      inhomogeneous steady state (to leading order) is given by
      \[
      \vy_i = \vy^* 
            +  \xi_i T \big[\vpi'(\vy^*) - \vdelta'(\vy^*)\big]^{-1}
	    \vpsi(\vy^*,\vy^*), \qquad i = 1,\ldots,n
      \]
      where the coefficients $\xi_i$ depend on the local properties of the cellular array,
      namely
      \[
       \xi_i = 1-\sum_{j \in \neigh_i} 
       \frac{1}{\vert \neigh_j \vert}.
      \]
  \end{enumerate}
\end{lem}
\begin{proof}
  If $T=0$ the statement is clearly true, so henceforth we assume $T\neq 0$.
  Since $\vpi(\vys) -
  \vdelta(\vys) = \mathbf{0}$, the right-hand side of~\eqref{eq:concBasedModel}
  vanishes for all $i$ if $\vert \neigh_i \vert = \vert \neigh_j \vert = \vert
  \neigh^* \vert$. On the other hand, if not all cells have the same number of
  neighbours and $T$ is small, we may seek for asymptotic steady states in the form
  $\vy_i = \vy^* + T \veta_i +
  \mathcal{O}(T^2)$ for $i=1,\dots,n$ and $(\veta_i)_j = \mathcal{O}(1)$. A
  Taylor expansion of the right-hand side around $(\vy^*,\ldots,\vy^*)^\textnormal{T}
  \in \R^{nm}$ gives, to leading order,
  \begin{equation}
  \mathbf{0} = 
  \vpi(\vy^*) - \vdelta(\vy^*) + T [\vpi'(\vy^*) - \vdelta'(\vy^*)]
  \veta_i +  T \sum_{j \in \neigh_i} \vnu_{ji}(\vy^*,\ldots,\vy^*) - \vnu_{ij}(\vy^*,\ldots,\vy^*),
  \label{eq:firstOrder}
  \end{equation}
  In order to find an expression for $\veta_i$, we evaluate the sums in
  \eqref{eq:firstOrder}:
  \begin{equation}
  \begin{split}
    \sum_{j \in \neigh_i} & \vnu_{ji}(\vy^*,\ldots,\vy^*) - \vnu_{ij}(\vy^*,\ldots,\vy^*)  \\       
    & = \sum_{j \in \neigh_i} \bigg(
                  \vpsi(\vy^*,\vy^*) \odot \vphi(\vy^*) \oslash \sum_{k \in \neigh_j} \vphi (\vy^*) 
                - \vpsi(\vy^*,\vy^*) \odot \vphi(\vy^*) \oslash \sum_{k \in \neigh_i} \vphi (\vy^*) 
		              \bigg) \\
    & = \vpsi(\vy^*,\vy^*) \sum_{j \in \neigh_i} 
    \bigg( 
    \frac{1}{\vert \neigh_j \vert} - \frac{1}{\vert \neigh_i \vert}
    \bigg), \\
    & = -\xi_i \vpsi(\vy^*,\vy^*)
  \end{split}
  \label{eq:geometry}
  \end{equation}
  and combining~\eqref{eq:firstOrder} with~\eqref{eq:geometry} we obtain the assert.
\end{proof}
\begin{rem}[Small-amplitude peak solutions]
  \label{rem:smallAmp}
  In finite regular arrays, cells in the interior have all the same number of
  neighbours, so we can use these properties to give formal definitions of
  interior and boundary sets
  \begin{align*}
  \mathcal{I} & = \left\{ i \in\left\{1,\ldots,n\right\} \left| 1 - \sum_{j \in \neigh_i}
    \frac{1}{\vert \neigh_j \vert} = 0 \right.\right\}, \\
    \mathcal{B} & = \left\{1,\ldots,n\right\} \setminus \mathcal{I}.
  \end{align*}
  In passing, we note that $\mathcal{B}$ contains in general more cells than the
  physical boundary. Lemma~\ref{lemma1} shows that, to leading order, steady states
  for small $T$ deviate from the homogeneous solutions only in $\mathcal{B}$, where
  auxin peaks and dips are proportional to $T$ and $\xi_i$, namely
  \begin{equation}
  \vy_i = 
  \begin{cases}
    \vy^* & \text{if $i \in \mathcal{I}$,} \\
    \vy^* +  \xi_i T \big[\vpi'(\vy^*) - \vdelta'(\vy^*)\big]^{-1}
    \vpsi(\vy^*,\vy^*) & \text{if $i \in \mathcal{B}$.} \\
  \end{cases}
  \label{eq:split}
  \end{equation}
\end{rem}
\begin{rem}[Irregular domains]
  Lemma 2.1 can not be applied in general if the domain is irregular, that is,
  if cells have different volumes and contact lengths: if, say, the active transport
  function $\vnu_{ij}$ depends explicitly on the cellular volume $V_i$ and the $V_i$
  are not all equal, then it is not possible to express $\vnu_{ij}$ as
  in~\eqref{eq:factorization}. However the theory can be extended to the case of
  irregular domains. We do not report this derivation here, but note that it
  features, as expected, cellular volumes and contact areas.  
\end{rem}

\subsubsection{One-dimensional domain and one component per cell}\label{sec:one-dimensional}
As an example, we consider the Smith model~\cite{Smith2006} with constant fixed PIN1
amount, posed on a one-dimensional array of identical cells with volume
$V$, Neumann boundary conditions at $i=1$ and free boundary conditions at $i=n$. The
Neumann boundary conditions are obtained by considering
ghost cells \cite{Draelants2013}, therefore boundary and interior sets are given by
$\mathcal{I} =
{1,\ldots,n-2}$ and $\mathcal{B} = \{n-1,n\}$, respectively. 
Furthermore we denote by $p$ the fixed PIN concentration and apply
Lemma~\ref{lemma1} with $m=1$, $y_i = a_i$ and
\begin{align*}
  & \pi \colon a \mapsto \frac{\rho_{_{\IAA}}}{1+\kappa_{_{\IAA}}a},
  & \delta \colon a \mapsto - \mu_{\IAA} a, \\
  & \psi \colon (a,b) \mapsto \frac{p}{V} \frac{a^2}{1+\kappa_T b^2},
  & \varphi \colon a \mapsto \exp(c_1 a),
\end{align*}
where all parameters are assumed to be strictly positive. By balancing production
and decay terms we find the positive homogeneous state
\[
a^* = \frac{-1+\sqrt{1+4\kappa_\IAA \rho_\IAA/\mu_\IAA}}{2\kappa_\IAA}.
\]
In the absence of active transport, $a^*$ is a stable steady state of the model since
$\pi'(a^*) - \delta'(a^*)<0$. For $0 \neq T \ll 1\, \mu\textrm{m}^3/\textrm{h}$,
$\vas$ is not a steady state since $\mathcal{B}$ is nonempty and
$\xi_{n-1} = -1/2$, $\xi_n = 1/2$, hence we obtain, to leading order
\[
a_i = 
\begin{cases}
  a^*  & \text{for $i = 1 \ldots n-2$,} \\[1em]
  a^* + 
  T\dfrac{p}{2V}
  \bigg[
  \dfrac{\rho_\IAA \kappa_\IAA}{1 + \kappa_\IAA (a^*)^2}
  + \mu_\IAA
  \bigg]^{-1}
  \dfrac{(a^*)^2}{1 + \kappa_T (a^*)^2}
  & \text{for $i = n-1$,} \\[1em]
  a^* 
  -T\dfrac{p}{2V}
  \bigg[
  \dfrac{\rho_\IAA \kappa_\IAA}{1 + \kappa_\IAA (a^*)^2}
  + \mu_\IAA
  \bigg]^{-1}
  \dfrac{(a^*)^2}{1 + \kappa_T (a^*)^2}
  & \text{for $i = n$.}
\end{cases}
\]

\subsubsection{One-dimensional domain and two components per cell}\label{sec:one-dimensional-2-comp}
The Smith model~\cite{Smith2006} features 2 ODEs per cell. If we pose this model on a
one-dimensional array of identical cells, the graph $G$ associated to the nodes
is the same as in our previous example, hence $\mathcal{B}$, $\mathcal{I}$ and $\vxi$
are unchanged. We can now apply Lemma~\ref{lemma1} with $m=2$, $\vy_i
=(a_i,p_i)^\textrm{T}$ and
    \begin{align*}
      & \vpi \colon 
	\begin{bmatrix} a \\ p \end{bmatrix}
	\mapsto 
	\begin{bmatrix} 
	\dfrac{\rho_\IAA}{1+\kappa_{\IAA}a} \\[1em] 
	\dfrac{\rho_{\PIN_0} + \rho_\PIN a}{1+\kappa_\PIN p} \\ 
	\end{bmatrix}, 
      & & \vdelta \colon 
	\begin{bmatrix} a \\ p \end{bmatrix}
	\mapsto 
	\begin{bmatrix} 
	- \mu_{\IAA} a \\
	- \mu_{\PIN} p \\
	\end{bmatrix},  
      \\
      & \vpsi \colon 
	\bigg(
	\begin{bmatrix} a \\ p \end{bmatrix},
	\begin{bmatrix} b \\ q \end{bmatrix}
	\bigg)
	\mapsto 
	\begin{bmatrix} 
	\dfrac{p}{V} \dfrac{a^2}{1+\kappa_T b^2} \\[1em]
	0
	\end{bmatrix},
      & & \vphi \colon 
      \begin{bmatrix} a \\ p \end{bmatrix}
      \mapsto 
      \begin{bmatrix} 
      \exp(c_1  a) \\
      0
      \end{bmatrix}.
    \end{align*}

Balancing production and decay terms we find a homogeneous strictly positive steady
state for $T=0$
\[
\vy^* = 
\begin{bmatrix} a^* \\ b^* \end{bmatrix} = 
\begin{bmatrix}
  \dfrac{-1+\sqrt{1+4\kappa_\IAA \rho_\IAA/\mu_\IAA}}{2\kappa_\IAA} \\[1em]
  \dfrac{-1+\sqrt{1+4\kappa_\PIN (\rho_{\PIN_0} + \rho_\PIN a^*)/\mu_\PIN}}{2\kappa_\PIN} \\
\end{bmatrix}
\]
which is stable since
\[
\spec( \vpi'(\vy^*)- \vdelta'(\vy^*) ) = 
\left\{
-\frac{\rho_\IAA \kappa_\IAA}{(1+\kappa_\IAA a^*)^2} - \mu_\IAA,
\;
-\frac{(\rho_{\PIN_0} + \rho_\PIN a^*) \kappa_\PIN}{(1+\kappa_\PIN a^*)^2} - \mu_\PIN
\right\}
\]
Since the parameters are assumed to be positive with the exception of
$\rho_{\PIN_0}$ which is nonnegative (see also Table ~\ref{table:parameters} )
we do not have a zero eigenvalue.

The inverse of $\vpi'(\vy^*)- \vdelta'(\vy^*)$ can be computed explicitly and for
$T \ll 1 \, \mu\textrm{m}^3/\textrm{h}$ we obtain to leading order
\[
a_i = 
\begin{cases}
  a^*  & \text{for $i = 1 \ldots n-2$,} \\[1em]
  a^* + 
  T \dfrac{p}{2V}
  \bigg[
  \dfrac{\rho_\IAA \kappa_\IAA}{1 + \kappa_\IAA (a^*)^2}
  + \mu_\IAA
  \bigg]^{-1}
  \dfrac{(a^*)^2}{1 + \kappa_T (a^*)^2}
  & \text{for $i = n-1$,} \\[1em]
  a^* 
  -T \dfrac{p}{2V}
  \bigg[
  \dfrac{\rho_\IAA \kappa_\IAA}{1 + \kappa_\IAA (a^*)^2}
  + \mu_\IAA
  \bigg]^{-1}
  \dfrac{(a^*)^2}{1 + \kappa_T (a^*)^2}
  & \text{for $i = n$,}
\end{cases}
\]
\[
p_i = 
\begin{cases}
  p^*  & \text{for $i = 1 \ldots n-2$,} \\[2.5em]
  p^* + 
  T\dfrac{p}{2V}
  \dfrac{
  \bigg[
    \dfrac{\rho_\PIN}{1 + \kappa_\PIN p^*}
  \bigg]
  \bigg[
  \dfrac{(a^*)^2}{1 + \kappa_T (a^*)^2}
  \bigg]
        }{
  \bigg[
    \dfrac{\rho_\IAA \kappa_\IAA}{(1 + \kappa_\IAA a^*)^2} + \mu_\IAA
  \bigg]
  \bigg[
    \dfrac{(\rho_{\PIN_0} + \rho_\PIN a^*) \kappa_\PIN}{(1 + \kappa_\PIN p^*)^2} + \mu_\PIN
  \bigg]
	}
  & \text{for $i = n-1$,} \\[2.5em]
  p^* - 
  T\dfrac{p}{2V}
  \dfrac{
  \bigg[
    \dfrac{\rho_\PIN}{1 + \kappa_\PIN p^*}
  \bigg]
  \bigg[
  \dfrac{(a^*)^2}{1 + \kappa_T (a^*)^2}
  \bigg]
        }{
  \bigg[
    \dfrac{\rho_\IAA \kappa_\IAA}{(1 + \kappa_\IAA a^*)^2} + \mu_\IAA
  \bigg]
  \bigg[
    \dfrac{(\rho_{\PIN_0} + \rho_\PIN a^*) \kappa_\PIN}{(1 + \kappa_\PIN p^*)^2} + \mu_\PIN
  \bigg]
	}
  & \text{for $i = n$.}
\end{cases}
\]

Another example of a concentration based transport model that features 2
ODEs per cell is the Chitwood model~\cite{Chitwood2012}. 
The model is given by 
    \begin{align*}
      & \vpi \colon 
	\begin{bmatrix} a \\ p \end{bmatrix}
	\mapsto 
	\begin{bmatrix} 
	\dfrac{\rho_\IAA}{1+\kappa_{\IAA}a} \\[1em] 
	\dfrac{\rho_{\PIN_0} + \rho_\PIN a}{1+\kappa_\PIN p} \\ 
	\end{bmatrix}, 
      & & \vdelta \colon 
	\begin{bmatrix} a \\ p \end{bmatrix}
	\mapsto 
	\begin{bmatrix} 
	- \mu_{\IAA} a \\
	- \mu_{\PIN} p \\
	\end{bmatrix},  
      \\
      & \vpsi \colon 
	\bigg(
	\begin{bmatrix} a \\ p \end{bmatrix},
	\begin{bmatrix} b \\ q \end{bmatrix}
	\bigg)
	\mapsto 
	\begin{bmatrix} 
	\dfrac{p}{V} \dfrac{\exp{c_2\ a} - 1 }{\exp{c_2\ b}} \\[1em]
	0
	\end{bmatrix},
      & & \vphi \colon 
      \begin{bmatrix} a \\ p \end{bmatrix}
      \mapsto 
      \begin{bmatrix} 
      \exp(c_1 a) \\
      0
      \end{bmatrix}.
    \end{align*}
For the same domain as above and with $m=2$ and $\vy_i
=(a_i,p_i)^\textrm{T}$ we can apply Lemma~\ref{lemma1}. 

The homogeneous steady state for $T=0$ is the same as for the Smith
model and for $T \ll 1 \, \mu\textrm{m}^3/\textrm{h}$ we obtain to leading order
\[
a_i = 
\begin{cases}
  a^*  & \text{for $i = 1 \ldots n-2$,} \\[1em]
  a^* + 
  T \dfrac{p}{2V}
  \bigg[
  \dfrac{\rho_\IAA \kappa_\IAA}{1 + \kappa_\IAA (a^*)^2}
  + \mu_\IAA
  \bigg]^{-1}
  \dfrac{\exp{(c_2 \ a^*)} -1}{\exp{(c_2 \ a^*)}}
  & \text{for $i = n-1$,} \\[1em]
  a^* 
  -T \dfrac{p}{2V}
  \bigg[
  \dfrac{\rho_\IAA \kappa_\IAA}{1 + \kappa_\IAA (a^*)^2}
  + \mu_\IAA
  \bigg]^{-1}
  \dfrac{\exp{(c_2 \ a^*)} -1}{\exp{(c_2 \ a^*)}}
  & \text{for $i = n$,}
\end{cases}
\]
\[
p_i = 
\begin{cases}
  p^*  & \text{for $i = 1 \ldots n-2$,} \\[2.5em]
  p^* + 
  T \dfrac{p}{2V}
  \dfrac{
  \bigg[
    \dfrac{\rho_\PIN}{1 + \kappa_\PIN p^*}
  \bigg]
  \bigg[
  \dfrac{(a^*)^2}{1 + \kappa_T (a^*)^2}
  \bigg]
        }{
  \bigg[
    \dfrac{\rho_\IAA \kappa_\IAA}{(1 + \kappa_\IAA a^*)^2} + \mu_\IAA
  \bigg]
  \bigg[
    \dfrac{(\rho_{\PIN_0} + \rho_\PIN a^*) \kappa_\PIN}{(1 + \kappa_\PIN p^*)^2} + \mu_\PIN
  \bigg]
	}
  & \text{for $i = n-1$,} \\[2.5em]
  p^* - 
  T \dfrac{p}{2V}
  \dfrac{
  \bigg[
    \dfrac{\rho_\PIN}{1 + \kappa_\PIN p^*}
  \bigg]
  \bigg[
  \dfrac{(a^*)^2}{1 + \kappa_T (a^*)^2}
  \bigg]
        }{
  \bigg[
    \dfrac{\rho_\IAA \kappa_\IAA}{(1 + \kappa_\IAA a^*)^2} + \mu_\IAA
  \bigg]
  \bigg[
    \dfrac{(\rho_{\PIN_0} + \rho_\PIN a^*) \kappa_\PIN}{(1 + \kappa_\PIN p^*)^2} + \mu_\PIN
  \bigg]
	}
  & \text{for $i = n$.}
\end{cases}
\]

\subsubsection{Two-dimensional domain of identical hexagonal cells}
Lemma~\ref{lemma1} also applies when the Smith or the Chitwood model
are posed on a two-dimensional array of identical hexagonal cells. In this case,
the computations for $\vy^*$ are identical to the previous example and the
asymptotic derivation is also straightforward. Instead of writing the full
expressions for $\mathcal{B}$, $\mathcal{I}$ and $\vxi$, we refer the reader to
Figure~\ref{fig:xiHexagons}, where the values of $\vxi$ are shown for two
corners of the domain. As claimed in Section \ref{subsec:TwoDimensionalDomain}, the highest peaks occur at
the top-left and right-bottom corners.

\subsection{Models with diffusion}
\label{subsec:modelsWithDiffusion}
We can extend the definition of concentration-based models to the case where
diffusion is present.
\begin{dfn}[Concentration-based model with diffusion]
 A concentration-based model with diffusion is a set of $mn$ ODEs of the form
\begin{equation}
  \dot{\vy}_i = \vpi(\vy_i) - \vdelta(\vy_i) + \mD \sum_{j \in \neigh_i} (\vy_j - \vy_i)
  + T \sum_{j \in \neigh_i} \vnu_{ji}(\vy_1,\ldots,\vy_n) - \vnu_{ij}(\vy_1,\ldots,\vy_n),
  \label{eq:concBasedModelDiff}
\end{equation}
for $i = 1,\ldots,n$, where $\mD \in \R^{m \times m}$ is a diagonal diffusion matrix and
all other quantities are as in Definition~\ref{def:conBasedModelNoDiff}
\end{dfn}
\begin{rem} Reasoning like in Lemma~\ref{lemma1}, we obtain $\vy_i = \vy^* + T
  \veta_i + \mathcal{O}(T^2)$ where $\veta_i$ satisfy
\[
\left[
\bm{J(\vys)}
+ \bm{L} \otimes \mD  \right]
\begin{bmatrix}
\veta_1 \\
\vdots \\
\veta_n
\end{bmatrix}
= 
\begin{bmatrix}
\xi_1 \vpsi(\vy^*,\vy^*) \\
\vdots \\
\xi_n \vpsi(\vy^*,\vy^*)
\end{bmatrix},
\]
$\bm{J}(\vy) \in \R^{mn \times mn}$ is block-diagonal with blocks
$\vpi'(\vy^*) - \vdelta'(\vy^*)$, $\bm{L} \in \R^{n \times n}$ is the Laplacian
matrix associated to the graph $G$  with Neumann boundary conditions and $\otimes$ denotes the Kronecker product between matrices.
The operator $\bm{L} \otimes \mD$ is negative
semi-definite and it has a zero eigenvalue, corresponding to a constant eigenvector.
However, summing this matrix to $\bm{J}(\vy) = \vpi'(\vy^*) - \vdelta'(\vy^*)$,
makes the resulting linear operator non-singular.
\end{rem}

In the presence of diffusion we can not directly apply the formula~\eqref{eq:split},
even for regular cellular arrays: owing to diffusion, cells in $\mathcal{I}$
will also deviate from the homogeneous state, hence peaks and dips are not
necessarily formed within $\mathcal{B}$, but may occur in interior cells that are
close to the boundary. First-order corrections for these cases can be computed
analytically using Chebyshev polynomials~\cite{Arfken1985} or numerically using
linear algebra routines. Even though we report below an example of this calculation,
we point out that in practice this is not necessary, since the numerical bifurcation
software gives access to the full nonlinear solution and to its linear stability.

\subsubsection{One-dimensional domain with diffusion and two components per
  cell}\label{sec:two-components-diffusion}
We return to the Smith model with $m=2$, posed on a row of identical cells, 
and we now add diffusion only in the auxin component.

The expressions for $a^*$ and $p^*$ are unchanged from
Section~\ref{sec:one-dimensional-2-comp}, as is the first order approximation of
$p_i$ (since there is no diffusion for $p$). Expressions for the first-order
approximations in $a_i$ are more
involved: proceeding as explained above for generic models with diffusion, we
obtain, to leading order
\begin{eqnarray}
\left[\left(\dfrac{\rho_{_{\IAA}}\kappa_{_{\IAA}}}{\left(1+\kappa_{_{\IAA}}a^*\right)^2} + \mu_{_{\IAA}}\right)
\left(
	\begin{array}{ccccc} 
				1 & 0 		 & \ldots & \ldots & 0  \\
				0 & 1 		 & 0 		  & \ldots & 0  \\
					&				 & \ddots & 			 &	 	\\
					&				 &				&	\ddots &	 	\\
				0 & \ldots & \ldots & 0			 & 1
	\end{array}
\right)
+ \dfrac{D}{V}
\left(
	\begin{array}{ccccc} 
				-1 & 1 		   & 				& 			 &   \\
				1  & -2 		 & 1 		  & 			 &   \\
					 &	\ddots & \ddots & \ddots &	 	\\
					 &				 &		1		&	-2 		 & 1 	\\
				   & 				 & 				& 1			 & -1
	\end{array}
\right)
\right] \nonumber \\
\cdot\left(\begin{array}{c} \eta_1 \\ \eta_2 \\ \vdots \\ \vdots \\ \eta_n \end{array}\right)
=
\left(\begin{array}{c} \xi_1  \psi\left(\vas,\vps\right) \\ \xi_2\psi_2\left(\vas,\vps\right) \\ \vdots \\ \vdots \\ \xi_n \psi_n\left(\vas,\vps\right) \end{array}\right)
=
\left(\begin{array}{c} 0 \\ \vdots \\ 0 \\
  \frac{1}{2}\psi_{n-1}\left(\vas,\vps\right) \\
  -\frac{1}{2}\psi_n\left(\vas,\vps\right) \end{array}\right).
\label{eq:linsystemDiffusion}
\end{eqnarray}

Solving this linear equation above led to the approximate solution
profiles in Figure~\ref{fig:FiniteDomainWithDiffusion} and the red 
solution branch in Figure~\ref{fig:analyticVsNumericalBD}.
The same derivation and figures can be obtained for the Chitwood model
(not shown).

\section*{Acknowledgments}
We acknowledge fruitful discussions with Gerrit T.S. Beemster, Jan Broeckhove, Dirk De Vos, Abdiravuf 
Dzhurakhalov, Etienne Farcot and Przemyslaw Klosiewicz.
DD acknowledges financial support from the Department of Mathematics
and Computer Science of the University of Antwerp. DA acknowledges the University of
Nottingham Research Development Fund, supported by the Engineering and Physical
Sciences Research Council (EPSRC). 


\end{document}